\documentclass[lettersize,journal]{IEEEtran}
\usepackage{amsmath,amsfonts}
\usepackage{amsthm}%
\usepackage{enumitem}%
\usepackage{algorithmic}
\usepackage{algorithm}
\usepackage{array}
\usepackage[caption=false,font =normalsize,labelfont=sf,textfont=sf]{subfig}
\usepackage{textcomp}
\usepackage{stfloats}
\usepackage{url}
\usepackage{verbatim}
\usepackage{graphicx}
\usepackage{cite}
\usepackage{framed}
\usepackage{epstopdf}
\usepackage{rotating}
\newtheoremstyle{thmstyleone}
{3pt} {3pt} {\itshape} { } {\bfseries} {.} {.5em} { }

\newtheoremstyle{thmstyletwo}
{3pt} {3pt} {\normalfont} { } {\bfseries} {.} {.5em} { }

\newtheoremstyle{thmstylethree}
{3pt} {3pt} {\normalfont} { } {\slshape} {.} {.5em} { }

\theoremstyle{thmstyleone}%
\newtheorem{theorem}{Theorem}
\newtheorem{proposition}[theorem]{Proposition}%

\theoremstyle{thmstyletwo}%
\newtheorem{assumption}{Assumption}[section]
\newtheorem{lemma}{Lemma}[section]
\theoremstyle{thmstylethree}%
\DeclareMathOperator*{\diag}{Diag}
\begin{document}

\title{ADMM-based Bilevel Descent Aggregation Algorithm for Sparse Hyperparameter Selection}
\author{Yunhai Xiao, Anqi Liu, Peili Li, and Yanyun Ding$^*$
	\thanks{Yunhai Xiao is with the Center for Applied Mathematics of Henan Province, Henan University, Zhengzhou, 450046, P.R. China (email: yhxiao@henu.edu.cn).}
	\thanks{Anqi Liu is with the School of Mathematics and Statistics, Henan University, Kaifeng, 475000, P.R. China (email: liuanqi\_an@163.com).}
	\thanks{Peili Li is with the Center for Applied Mathematics of Henan Province, Henan University, Zhengzhou, 450046, P.R. China, and also with the School of Mathematics and Statistics, Henan University, Kaifeng, 475000, P.R. China (email: lipeili@henu.edu.cn).}
	\thanks{Yanyun Ding is with the Institute of Applied Mathematics, Shenzhen Polytechnic University, Shenzhen, 518055, P.R. China (email: dingyanyun@szpu.edu.cn).}
}

\maketitle

\begin{abstract}
It is widely acknowledged that hyperparameter selection plays a critical role in the effectiveness of sparse optimization problems.
The bilevel optimization provides a robust framework for addressing this issue, but these existing methods depend heavily on the  lower-level singleton (LLS) assumption, which greatly limits their practical applicabilities.
To tackle this technical challenge, this paper focus on a particular type of nonsmooth convex sparse optimization problem and presents a new bilevel optimization framework. This framework effectively integrates the alternating direction method of multipliers (ADMM) with a bilevel descent aggregation (BDA) algorithm.
Specifically, it employs ADMM to efficiently address the lower-level problem and uses BDA to explore the hyperparameter space, thereby integrating both the upper and lower-level problems.
It is important to emphasize that a key contribution of this paper lies in the presentation of a novel convergence analysis. The analysis illustrates that the proposed ADMM-BDA algorithm achieves global convergence under significantly relaxed conditions, thereby departing from the LLS assumption that are often required in the literature.
We conduct a series of numerical experiments utilizing synthetic and real-world data, and do performance comparisions against some state-of-the-art algorithms. The results indicates that ADMM-BDA exhibits superior effectiveness and robustness for solving bilevel programming problems, especially when the lower-level problem is an elastic-net penalized statistics problem.
\end{abstract}

\begin{IEEEkeywords}
Bilevel programming; Hyperparameter selection; Bilevel descent aggregation; Alternating direction method of multipliers; Lower-level singleton.
\end{IEEEkeywords}

\section{Introduction}\label{intro}
\IEEEPARstart{I}{t} is well-known that the sparse optimization plays a pivotal role in fields such as signal processing \cite{CT2005}, statistics \cite{ANW2012}, and machine learning \cite{LHL2020}, etc. It aims to identify solutions to mathematical optimization problems where the resulting solution vector is inherently sparse.
The sparse solution is defined as a vector in which the vast majority of elements are zero, with only a minimal number of non-zero entries.
In sparse optimization problems, let $x^t \in \mathbb{R}^n$ be a true sparse vector and $b \in \mathbb{R}^m$ denote an observed data, where $m \ll n$. The observation data typically follows the form:
$$
b = A x^t + \varepsilon,
$$
where $A \in \mathbb{R}^{m \times n} $ mapping the original vector $x^t$ to the observed data $b\in\mathbb{R}^m$, and $\varepsilon \in \mathbb{R}^m$ represents noise or errors, such as Gaussian or heavy-tailed noise, that are inevitably contained during the measuring process \cite{DZL2023}.
However, all of these pose a challenge when it comes to accurately recovering the sparse vector $x^t $.

The goal of sparse optimization is to find a sparse solution $ \bar{x} $ that approximates $ x^{t} $ using some appropriate sparse optimization models and algorithms. The sparse optimization problem is typically formulated as:
\begin{equation}\label{LL}
	\min_{x\in\mathbb{X}} \left\{ \mathcal{Q}(Ax-b) + \sum_{i=1}^{r} \lambda_i \mathcal{R}_i(x) \right\},
\end{equation}
where $\mathbb{X}\subseteq\mathbb{R}^n$ is a convex compact set, $ \mathcal{Q}(\cdot): \mathbb{R}^m \to \mathbb{R} $ is the loss function measuring the error between the model output and actual observations, and $ \mathcal{R}_i(\cdot): \mathbb{R}^n \to \mathbb{R}, (i = 1, \ldots, r) $ are penalty functions that promote sparsity (e.g., $ \ell_1 $-norm) or stability (e.g., $ \ell_2 $-norm). The  parameter $ \lambda = (\lambda_1, \dots, \lambda_r) \in \mathbb{R}^{r}_+ $ is a pivotal hyperparameter that controls the solution's properties \cite{CF2022, BPG2007}.

When selecting hyperparameters, the traditional methods such as grid-search and random-search are poorly suited for sparse optimization problems due to their undirected approach. The absence of guidance in these methods makes them ineffective at handling non-smooth and sparse structures, which leads to higher computational costs  \cite{OTK2021,GYY2022,KP2013,KBH2008,WL2024}.
Consequently, the shortcomings of these traditional methods make it essential to develop a bilevel optimization framework \cite{LLZ2022}, which characterizes the relationships between the hyperparameters and other   parameters.
In this framework, an upper-level problem is constructed to minimize the validation loss for identifying the optimal hyperparameters, while the lower-level  problem is designed to efficiently solve the sparse model \eqref{LL} and obtain sparse solutions.
As demonstrated in the literature that, this bilevel approach consistently achieves superior accuracy, faster convergence, and better scalability compared to other traditional techniques \cite{MBB2009,BHJ2006}.
Mathematically, the bilevel optimization formulation for selecting hyperparameters is expressed as follows:
\begin{equation}\label{model0}
	\begin{aligned}
		\min_{x \in \mathbb{X}, \lambda \in \Lambda} \enskip & \mathcal{F}(x, \lambda) \\
		\text{s.t.} \enskip & x \in \arg\min_{x \in \mathbb{X}} \left\{ \mathcal{Q}(Ax - b) + \sum_{i=1}^{r} \lambda_i \mathcal{R}_i(x) \right\}.
	\end{aligned}
\end{equation}
where $\mathbb{X}\subseteq \mathbb{R}^n$ and $\Lambda\subseteq \mathbb{R}_+^r$ are convex and compact sets, and $\mathcal{F}$ is a convex outer-objective function (e.g., mean squared error (MSE)) that is always Lipschitz continuous.
The upper-level problem seeks to minimize the validation error with respect to the hyperparameter $\lambda$, whereas the lower-level problem aims to obtain a   solution $x$ for a fixed   $\lambda$.
For $\mathcal{Q}(\cdot)$, selecting it as a least squares function yields a continuously differentiable function, while choosing an $\ell_q$-norm function (with $q = 1, 2, \infty$) results in a non-smooth function.

It is straightforward to see that the inherent hierarchical coupling in \eqref{model0} poses a significant theoretical and computational challenge to the problem of jointly optimizing hyperparameter $\lambda$ and parameter $x$ \cite{KP2013,OTK2021,GYY2022,SMK2024}.
Some existing approaches focus primarily on the singleton lower-level solution case \cite{FFS2018, SCC2019}, where the lower-level problem is strongly convex under some regularity conditions.
This assumption enables the unique solution to be incorporated into the upper-level problem to result in a single-level problem \cite{P2016, RFS2019, FDF2017}.
When the penalty term in the lower-level problem involves an elastic-net \cite{ZH2005} or Lasso \cite{BT2009}, the uniqueness of the lower-level solution is inherently compromised.
By relaxing the assumption of a single lower-level solution, some advanced techniques such as the implicit gradient-based joint optimization \cite{FS2018} and the implicit forward differentiation optimization \cite{BKB2020, BKM2022} are capable of selecting hyperparameters. However, there still exists a notable theoretical gap in the convergence analysis towards some stationary solutions.
This issue has driven significant research into algorithms that guarantee convergence without relying on the assumptions of strong convexity or smoothness in the lower-level problem.

It is important to highlight that the bilevel descent aggregation (BDA) algorithm introduced by Liu et al. \cite{LMY2020} leverages gradient information from both levels simultaneously.
Crucially, the BDA algorithm ensures convergence even without requiring a unique lower-level solution, and offers a theoretically guaranteed approach for solving bilevel optimization problems.
However, the BDA algorithm still requires the lower-level problem to satisfy smoothness conditions in its algorithmic framework.
In a subsequent work, Liu et al. \cite{LLZ2022} proposed a bilevel meta-optimization (BMO) algorithm based on the fixed-point framework, which facilitates the simultaneous optimization of both training variables and hyper-training variables.
However, the BMO fails to fully exploit the separable structure of the lower-level problem in \eqref{model0}.
Therefore, the developments of other algorithms that adequately use the favorable structure of the lower-level problem are highly necessary.

In this paper, motivated by the limitations of the BDA algorithm, we investigate numerical methods for the bilevel optimization \eqref{model0}, where the   lower-level problem  lacks the unique solution assumption and does not satisfy assumptions of smoothness or strong convexity.
Specifically, we apply the alternating direction method of multipliers (ADMM) within the BDA  framework which consists two key innovations: first, it establishes a collaboration between the upper-level hyperparameters and the lower-level sparse solution, and second, it efficiently solves the lower-level problem by using the separable structure and proximal mappings.
In other words, the primary contribution of this paper is the integration of ADMM into the BDA framework, along with  its convergence results that does not rely on the uniqueness of the solution or strong convexity assumptions for the lower-level problem.
Theoretically, we prove that any limit point of the sequence generated by the algorithm is a solution to \eqref{model0}, where the lower-level variable satisfies optimal condition, and the upper-level variable minimizes the upper-level outer-objective function.
Additionally, we prove that as the iterations increases, the optimal value of the upper-level problem converges to the optimum value of \eqref{model0}.
In our experiments, we   perform comprehensive numerical tests where the lower-level problem incorporates an elastic-net or a generalized-elastic-net penalty to demonstrate the effectiveness of our proposed algorithm in solving bilevel optimization problem.

The remaining part of this paper is organized as follows: Section \ref{preli} reviews some preliminary results used in the subsequent developments.
Section \ref{sec2} develops the full steps of the ADMM-based BDA algorithm (abbr. ADMM-BDA).
Section \ref{sec3} presents a convergence analysis of ADMM-BDA.
Section \ref{num} conducts some numerical experiments on elastic-net and generalized-elastic-net peanlized lower-level problem, and do performance comparisons with some classical hyperparameter selection methods.
Section \ref{con} concludes the paper with a summary.
\section{Preliminaries}\label{preli}

In this section, we quickly review some preliminary results used in the subsequent developments.
Let $\mathbb{R}^n$ be an $n$-dimensional Euclidean space equipped with an inner product $\langle \cdot, \cdot \rangle$ and an induced norm $\|\cdot\|_2$.
The inner product induced by a self-adjoint semi-positive linear operator $\mathcal{S}$ is defined as
$\langle x, y \rangle_{\mathcal{S}} := \langle \mathcal{S}x, y \rangle$,
and the corresponding norm is given by
$$\|x\|^2_{\mathcal{S}} := \langle x, x \rangle_{\mathcal{S}} = x^\top \mathcal{S} x.$$
Let $\mathbb{C}$ be a closed convex set. The distance from a point $x$ to the set $\mathbb{C}$ with respect to the $\mathcal{S}$-norm, denoted as $\mathrm{dist}_{\mathcal{S}}(x, \mathbb{C})$, is defined as
$$\mathrm{dist}_{\mathcal{S}}(x, \mathbb{C}) = \inf_{\bar{x} \in \mathbb{C}} \big\{\|\bar{x} - x\|_{\mathcal{S}}\big\},$$
which represents the infimum of the distances from $x$ to all points in $\mathbb{C}$.

Let $f: \mathbb{R}^n \to (-\infty, +\infty]$ be a closed, proper, and convex function.
A vector $x^*$ is said to be a subgradient of $f$ at point $x$ if $f(z)\geq f(x)+\langle x^*,z-x\rangle$ for all $z\in\mathbb{R}^n$. The set of all subgradients of $f$ at $x$ is called the subdifferential of $f$ at $x$ and is denoted by $\partial f(x)$. Obviously, $\partial f(x)$ is a closed convex set while it is not empty.
The Fenchel conjugate of $f$ at $x$ is defined as $f^{*}(x):=\sup _{y \in \mathbb{R}^{n}}\{\langle x, y\rangle-f(y)\}$. The Moreau envelope function and the proximal mapping of $f$ with a parameter $t > 0$ are defined, respectively, as follows
$$
\begin{aligned}
	&\Phi_{t f}(x) :=\min _{y \in \mathbb{R}^{n}}\Big\{tf(y)+\frac{1}{2}\|y-x\|_2^{2}\Big\},\\
	&\operatorname{Prox}_{tf}(x) :=\underset{y \in \mathbb{R}^{n}}{\operatorname{argmin}}\Big\{tf(y)+\frac{1}{2}\|y-x\|_2^{2}\Big\}.
\end{aligned}
$$
Let $\Pi^{\mathcal{S}}_{\mathbb{C}}(\cdot)$ be the projection operator onto the set $\mathbb{C}$ associated with linear operator $\mathcal{S}$, defined by
$$
\Pi^{\mathcal{S}}_{\mathbb{C}}(x) = \arg \min_{\bar{x} \in \mathbb{C}} \big\{\| \bar{x} - x \|^2_{\mathcal{S}}\big\}.
$$
When $\mathcal{S}$ is an identity matrix, it reduces to the traditional projection operator $\Pi_{\mathbb{C}}(\cdot)$.
In convex optimization literature, it is known that that, for the $\ell_q$-norm function with $q=1$, $2$, and $\infty$, its proximal mappings enjoys closed-form expressions, see \cite{R2015,LST2022}.
Let $p>0$ and $q >0$ such that $1/p+1/q=1$.
Let $f(x):=t\|x\|_q$ with a positive scalar $t$, then it is trivial to deduce that $f^{*}(x)=\delta_{\mathbb{B}^{(t)}_{p}}(x)$ where ${\mathbb{B}^{(t)}_{p}}:=\{x \ | \ \|x\|_{p}\leq t\}$ and $\delta_\mathbb{C}(x)$ is an indicator function such that $\delta_\mathbb{C}(x)=0$ if $x\in \mathbb{C}$ and $+\infty$ otherwise.
When $f=\|x\|_1$, it is easy to deduce that
$$
\operatorname{Prox}_{tf}(x)=x-\Pi_{\mathbb{B}^{(t)}_{\infty}}(x), \enskip
\Pi_{\mathbb{B}^{(t)}_{\infty}}(x)=\max\{-t,\min\{x,t\}\}.
$$
When $f(x):=\|x\|_2$,  it gets that $f^{*}(x)=\delta_{\mathbb{B}^{(t)}_{2}}(x)$  and that
$$
\operatorname{Prox}_{tf}(x)=x-\Pi_{\mathbb{B}^{(t)}_{2}}(x),\enskip
\Pi_{\mathbb{B}^{(t)}_{2}}(x)=\min\{t,\|x\|_2\}\frac{x}{\|x\|_2}.
$$
When $f(x):=\|x\|_{\infty}$, it has that $f^{*}(x)=\delta_{\mathbb{B}^{(t)}_{1}}(x)$ and that
$$
\begin{aligned}
	&\operatorname{Prox}_{tf}(x)=x-\Pi_{\mathbb{B}^{(t)}_{1}}(x),\\
	&\Pi_{\mathbb{B}^{(t)}_{1}}(x)=\left\{\begin{array}{ll}
		x^*, & \text { if }  \|x\|_1\leq t,\\
		\mu P_{x} \Pi_{\Delta_{n}}\left(P_{x} x /t\right), & \text { if }  \|x\|_1>t,
	\end{array}\right.
\end{aligned}
$$
where $P_{x}:=\diag(\operatorname{Sign}(x))$ and $\Pi_{\Delta_{n}}(\cdot)$ denotes the projection onto the simplex $\Delta_{n}:=\{x \in \mathbb{R}^{n} \mid e_{n}^{\top} x=1, x \geq 0\}$, in which $\diag(\cdot)$ denotes a diagonal matrix with elements of a vector on its diagonal positions.

\section{ADMM-BDA for solving \eqref{model0}}\label{sec2}
In the following text, for clarity, we use superscript `$k$' to denote the outer iteration and superscript `(j)' to denote the inner iteration.
Assuming that at the $k$-th iteration, the hyperparameter $\lambda^k\in \Lambda$ is given.
The ADMM-BDA algorithm discussed below consists of two steps: first, an ADMM step, followed by a BDA step.

At the first step, we focus on the convex and nonsmooth lower-level problem  in \eqref{model0}.
For convenience, we use an auxiliary variable $y:=Ax-b$, then  the lower-level problem in \eqref{model0} is rewritten as
\begin{equation}\label{lmodel}
	\begin{array}{cl}
		\min\limits_{x\in\mathbb{R}^n, y\in\mathbb{R}^m}&\quad   \mathcal{Q}(y)+\sum^r_{i=1}\lambda_i^k\mathcal{R}_i(x) \\[3mm]
		\text{s.t.} &\quad  Ax-y=b.
	\end{array}
\end{equation}
Let $\sigma>0$ be a penalty parameter, the augmented Lagrangian function associated with problem   \eqref{lmodel}  is given by
$$
\begin{aligned}
	\mathcal{L}_{\sigma}(x,y; z):= &\mathcal{Q}(y)+\sum^r_{i=1}\lambda_i^k\mathcal{R}_i(x)\\
	&+\langle z,Ax-y-b\rangle+\frac{\sigma}{2}\|Ax-y-b\|^2_{2},
\end{aligned}
$$
where $z\in\mathbb{R}^{m}$ is  multiplier associated with the constraint.
We employ ADMM to iteratively minimize the augmented Lagrangian function with respect to the disjoint variables $x$ and $y$ in a Gauss-Seidel manner, progressing in the sequence $y \to z \to x$.
Starting from $(x_l^{(0)}, y_l^{(0)}, z_l^{(0)})$, where the subscript `$l$' denotes that the variables are associated with the lower-level problem.
Given $(x^{(j)}, y_l^{(j)}, z_l^{(j)})$, the   ADMM   generates the next point $(x_l^{(j+1)}, y_l^{(j+1)}, z_l^{(j+1)})$ based on the following framework:
\begin{align}
	{y}_l^{(j+1)}&=\arg\min_{y\in\mathbb{R}^m}\Big\{\mathcal{Q}(y)+\frac{\sigma}{2}
	\|Ax^{(j)}-y-b+{z}_l^{(j)}/\sigma\|_2^2\notag\\[2mm]
	&\quad+\frac{\sigma}{2}\|y-{y}_l^{(j)}\|_{\mathcal{S}_y}^2\Big\},\label{suby0}\\[2mm]
	{z}_l^{(j+1)}&={z}_l^{(j)}+\sigma\Big(Ax^{(j)}-{y}_l^{(j+1)}-b\Big),\label{subz}\\[2mm]
	{x}_l^{(j+1)}&=\arg\min_{x\in\mathbb{R}^n}\Big\{\sum^r_{i=1}\lambda_i^k\mathcal{R}_i(x)+\frac{\sigma}{2}
	\|Ax-{y}_l^{(j+1)}-b\notag\\[2mm]
	&\quad+{z}_l^{(j+1)}/\sigma\|_2^2+\frac{\sigma}{2}\|x-x^{(j)}\|_{\mathcal{S}_x}^2\Big\},\label{subx0}
\end{align}
where $\mathcal{S}_x$ and $\mathcal{S}_y$ are self-adjoint semi-positive linear operators.
It is important to note that the notation $x^{(j)} $ is used here instead of $x_l^{(j)} $, as it is also associated with the upper-level problem described below.

Choose $\mathcal{S}_y:=\eta \mathcal{I}$ with $\eta\geq 0$, and choose $\mathcal{S}_x:=\zeta \mathcal{I}-A^{\top}A$ with $\zeta>0$ be a positive scalar such that $\mathcal{S}_x$ be positive definite. Denote $\mathcal{R}(\cdot,\lambda):=\sum^r_{i=1}\lambda_i\mathcal{R}_i(\cdot)$.
It is a trivial task to deduce that the $y$- and $x$-subproblems can be written as the following proximal mapping forms:
\begin{equation}\label{suby}
	{y}_l^{(j+1)}=\operatorname{Prox}_{({\sigma}+\sigma\eta)^{-1}\mathcal{Q}(\cdot)}\bigg(\frac{\sigma(Ax^{(j)}-b)+{z}_l^{(j)}+\sigma\eta {y}_l^{(j)}}{\sigma(1+\eta)}\bigg),
\end{equation}
and
\begin{equation}\label{subx}
	\begin{aligned}
		&{x}_l^{(j+1)}=\operatorname{Prox}_{(\zeta\sigma)^{-1}\mathcal{R}(\cdot,\lambda^k)}\bigg(\frac{A^\top({y}_l^{(j+1)}+b-{z}_l^{(j+1)}/\sigma)}{\zeta}\\
		&\qquad\qquad+\frac{(\zeta I-A^\top A)x^{(j)}}{\zeta}\bigg),
	\end{aligned}
\end{equation}
which implies that the $x$- and $y$-subproblems are easily implemented when $\mathcal{Q}(\cdot)$ and $\mathcal{R}(\cdot)$ are both simple functions, such as the norm functions $\ell_q$-norm with $q=1,2,\infty$, or certain quadratic functions.

At the second step, we apply the BDA algorithmic framework to the temporary point $x_l^{(j+1)}$generated by ADMM for the lower-level problem.
Note that the upper-level problem $\mathcal{F}(x, \lambda)$ is assumed to be convex and Lipschitz continuous.
Given $x^{(j)}$, as noted by Liu et al. \cite{LMY2020}, a gradient-based point can be computed with an appropriately chosen step size as follows:
\begin{equation}\label{xj1u}
	x_{u}^{(j+1)}=x^{(j)}-s_j \mathcal{S}^{-1}\nabla_x \mathcal{F}(x^{(j)},\lambda^k),
\end{equation}
where $s_j=s/{(j+1)}$ with a positive scalar $s$, and $\mathcal{S}:=\sigma(A^\top A+\mathcal{S}_x)$ is a positive definite operator.
Using $x_l^{(j+1)}$ and $x_u^{(j+1)}$, the updated point $x^{(j+1)}$ is then computed as follows:
\begin{equation}\label{xj1}
	x^{(j+1)} = \Pi^{\mathcal{S}}_{\mathbb{X}}\left(\mu x_u^{(j+1)} + (1-\mu) x_l^{(j+1)}\right),
\end{equation}
where $\mu\in(0,1)$ and $\Pi^{\mathcal{S}}_{\mathbb{X}}(\cdot)$ denotes the projection operator onto the set $\mathbb{X}$ with operator $\mathcal{S}$.
This formula shows that $x^{(j+1)}$ is a convex combination of the point $x_l^{(j+1)}$ from the lower-level problem and the point $x_u^{(j+1)}$ from the upper-level problem, which conducts the connections of the bilevel problems.

After the iteration is complete, the hyperparameter $\lambda^k $ is updated accordingly before moving on to the next iteration.
Based on the previous analysis, we can summarize the ADMM-BDA for the bilevel programming problem \eqref{model0} as follows:
\begin{framed}
	\noindent
	{\bf Algorithm: ADMM based BDA for \eqref{model0} (ADMM-BDA)}
	\vskip 1.0mm \hrule \vskip 1mm
	\noindent
	\begin{itemize}[leftmargin=10mm]
		\item[Step $0$.] Let $\mathcal{S}_y:=\eta \mathcal{I}$ and $\mathcal{S}_x:=\zeta \mathcal{I}-A^{\top}A$ where $\eta\geq 0$ and $\zeta>0$ be positive scalars such that $\mathcal{S}_y$ and $\mathcal{S}_x$ be positive definite. Let $\mathcal{S}:=\sigma (A^{\top}A +\mathcal{S}_x)$. Input positive constants $\sigma$, $\alpha$, $s$,  $\mu\in(0,1)$, and a positive integer sequence $\{J_k\}$.  Initialize $\lambda^0$ and let $k:=0$.
		\item[Step $1$.] Given $\lambda^k$, output $x^{(J_k)}$ by the following steps. Initialize $({y}_l^{(0)}, {z}_l^{(0)},x^{(0)}_l)$, for $j=1,...,J_k$, do the following operations iteratively:
		\begin{itemize}[leftmargin=10mm]
			\item [Step $1.1$.] Compute ${y}_l^{(j+1)}$, ${z}_l^{(j+1)}$, and  ${x}_l^{(j+1)}$  according to \eqref{suby}, \eqref{subz}, and \eqref{subx}, respectively;
			\item [Step $1.2$.] Compute $x_{u}^{(j+1)}$ according to \eqref{xj1u},
			where $\{s_j\}$ is a step-size with updating rule $s_j={s}/{(j+1)}$;
			\item [Step $1.3$.] Compute $x^{(j+1)}$ according to \eqref{xj1}.
		\end{itemize}
		\item[Step $2$.] Let $x^{k+1}:=x^{(J_k)}$ and compute
		$$
		\lambda^{k+1} \in \arg\min_{\lambda \in \Lambda} \mathcal{F}\Big(x^{k+1},\lambda\Big).
		$$
		\item[Step $3$.] Let  $k:=k+1$, go to Step $1$.
	\end{itemize}
	
\end{framed}

It is important to note that the $\lambda$-updating step can be quite restrictive to implement in practical applications. Specifically, we indicate that if the outer objective function $\mathcal{F}(x,\cdot)$ is a quadratic function with a fixed $x$, it can be solved exactly. However, in other cases, it may be more effective to take a gradient step, such as
$$
\lambda^{k+1}=\Pi_{\Lambda}^{\mathcal{S}}\Big(\lambda^{k}-\alpha\nabla_\lambda\mathcal{F}(x^{k+1},\lambda^k)\Big),
$$
which has demonstrated good performance, as evidenced by our numerical experiments.
 \section{Convergence analysis}\label{sec3}
 In this section, we derive the convergence result for the sequence $\{x^k, \lambda^k\}$ generated by the ADMM-BDA algorithm. The proof approach is inspired by Liu et al. \cite{LLZ2022}, but unlike their approach for lower-level problem, we utilize ADMM to fully exploit the separable structure of problem \eqref{model0} to produce an approximated solution.
 
 For clarity, we restate the bilevel optimization problem \eqref{model0} below, where the lower-level problem is assumed to be both nonsmooth and non-strongly convex:
 
 \begin{equation}\label{model1}
 	\begin{aligned}
 		\min_{x\in\mathbb{X}, \lambda\in \Lambda} \quad & \mathcal{F}(x,\lambda)\\
 		\text{s.t.} \quad& x \in \mathcal{V}(\lambda),
 	\end{aligned}
 \end{equation}
 where
 $$
 \mathcal{V}(\lambda):=\arg\min_{x\in \mathbb{R}^n}\Big\{\mathcal{Q}(Ax-b)+\sum^r_{i=1}\lambda_i\mathcal{R}_i(x)\Big\}.
 $$
 To facilitate the theoretical analysis, we introduce an auxiliary function $\varphi(\lambda)$ with
 $$
 \varphi(\lambda):=\inf_{x\in \mathcal{V}(\lambda)} \{\mathcal{F}(x,\lambda)\}.
 $$
 In fact, the execution of ADMM-BDA is equivalent to solving the following problem:
 \begin{equation}\label{model2}
 	\min_{\lambda\in \Lambda}\ \ \Big\{\varphi_{J}(\lambda):=\mathcal{F}\Big(x^{(J)}({\lambda}),\lambda\Big)\Big\},
 \end{equation}
 where $x^{(J)}({\lambda})$ is  generated by Step $1$ of ADMM-BDA with a fixed hyperparameter $\lambda$ and a fixed number of steps $J$.
 Moreover, from Step $1.3$, we see that the point $x^{(j+1)}$ is also parameterized by $\lambda$, specifically:
 \begin{equation}\label{xlambda}
 	x^{(j+1)}(\lambda) = \Pi^{\mathcal{S}}_{\mathbb{X}}\left( \mu x_u^{(j+1)}(\lambda) + (1-\mu) x_l^{(j+1)}(\lambda) \right),
 \end{equation}
 where both $x_u^{(j+1)}$ and $x_l^{(j+1)}$ are all parameterized by $\lambda$.
 
 Based on the auxiliary function $\varphi(\lambda)$, we next demonstrate that any limit point $(\bar{x}, \bar{\lambda})$ of the sequence $\left\{\big(x^{(J_k)}, \lambda^{k}\big)\right\}$, where {$\lambda^{k} \in \arg\min_{\lambda \in \Lambda} \{\varphi_{J_k}(\lambda)\}$}, generated by  ADMM-BDA, is a solution to the bilevel problem \eqref{model1}. In other words, we can obtain the solution to \eqref{model1} by solving the problem \eqref{model2}.
 For this purpose, we need to make the following basic assumptions.
 \begin{assumption}\label{ass2}
 	Assume that both $\mathbb{X}$ and $\Lambda$ are convex and compact sets. Suppose that the solution set of the lower-level problem $\mathcal{V}(\lambda)$ is nonempty.  The functions $\mathcal{Q}(\cdot)$ and $\mathcal{R}_i(\cdot)$ in the lower-level problem are proper, closed, convex functions.
 \end{assumption}
 
 \begin{assumption}\label{ass1}
 	Assume that the upper-level out-objective function $\mathcal{F}(x,\lambda)$ is continuous on $\mathbb{X}\times \Lambda$. For any $\lambda\in\Lambda$, $\mathbb{F}(\cdot,\lambda): \mathbb{X}\rightarrow \mathbb{R}$ is $L_{\delta}$-smooth, which means that the gradient of $\mathcal{F}(\cdot,\lambda)$, $\nabla_{x}\mathcal{F}$ is Lipschitz continuous with $L_\delta$ as its Lipschitz constant. Meanwhile, assume that the function $\mathcal{F}(\cdot,\lambda)$ is convex and has $F_0$ as its lower bound.
 \end{assumption}
 
 It should be noted that the above assumptions are easily satisfied. An example is the MSE out-objective function, which is convex with respect to its input and has a lower bound of zero.  These assumptions are also commonly used in some existing approaches, see e.g., \cite{ZLZ2021,GYY2022}. Next, we will utilize a lemma to prove the non-expansive property of the iterative procedure \eqref{suby0}-\eqref{subx0} for any $\lambda\in \Lambda$.
 
 \begin{lemma}\label{lemma1}
 	Suppose Assumption \ref{ass2} holds. Given $\lambda^k\in \Lambda$,  define $\omega_l^{j}:=({y}_l^{(j)},{z}_l^{(j)},x_l^{(j)})$ and $\omega^{j}=({y}_l^{(j)},{z}_l^{(j)},x^{(j)})$. For all $j\geq 0$, and the $y$-, $z$-, and $x$-subproblems associated with \eqref{suby0}-\eqref{subx0}, we get that:\\
 	\begin{enumerate}
 		\item[(a)]  The iterative format \eqref{suby0}-\eqref{subx0}
 		can be expressed in the following compact form:
 		$$\omega_l^{j+1}=\mathcal{\tilde{T}}(\omega^{j}),$$
 		where $\mathcal{\tilde{T}}=(\mathcal{\tilde{D}}+\mathcal{\tilde{S}})^{-1} \mathcal{\tilde{S}}$ with
 		\begin{align*}
 			\mathcal{\tilde{D}} &:=\left(\begin{array}{c}
 				\partial_y \mathcal{Q}(y)-z \\
 				b-Ax+y\\
 				\partial_x \mathcal{R}(x,\lambda)+A^{\top}z
 			\end{array}\right) \\
 			\mathcal{\tilde{S}} &:=\left(\begin{array}{ccc}
 				\sigma\mathcal{S}_y & 0 & 0 \\
 				0 & \sigma^{-1} \mathcal{I} & A\\
 				0 & A^{\top} & \sigma (A^{\top}A +\mathcal{S}_x)
 			\end{array}\right).
 		\end{align*}

 		\item[(b)]  The operator $\mathcal{\tilde{T}}$ is firmly non-expansive with respect to the norm $\|\cdot\|_{\mathcal{\tilde{S}}}$, i.e., for any $(u,v)\in \text{dom}\ \mathcal{\tilde{T}}\times\text{dom}\ \mathcal{\tilde{T}}$, it holds that
 		$$
 		\|\mathcal{\tilde{T}}(u)-\mathcal{\tilde{T}}(v)\|^2_{\mathcal{\tilde{S}}}\leq\big\langle u-v, \mathcal{\tilde{T}}(u)-\mathcal{\tilde{T}}(v)\big\rangle_{\mathcal{\tilde{S}}}.
 		$$
 		\item[(c)]  The operator $\mathcal{\tilde{T}}$ is closed, i.e., its graph $\text{gph}\ (\mathcal{\tilde{T}})$ is closed.
 	\end{enumerate}
 \end{lemma}
 \begin{proof}
 	(a) According to the first-order optimality conditions,  the \eqref{suby0}-\eqref{subx0} subproblems can be equivalently reformulated as:
 	$$
 	\left\{
 	\begin{array}{cl}
 		0&\in \partial_y\mathcal{Q}({y}_l^{(j+1)})-\sigma
 		\big(Ax^{(j)}-{y}_l^{(j+1)}-b+{z}_l^{(j)}/\sigma\big)\\[2mm]
 		&\quad+\sigma\mathcal{S}_y({y}_l^{(j+1)}-{y}_l^{(j)}),\\[2mm]
 		{z}_l^{(j+1)}&={z}_l^{(j)}+\sigma\big(Ax^{(j)}-{y}_l^{(j+1)}-b\big),\\[2mm]
 		0&\in\partial_x \mathcal{R}({x}_l^{(j+1)},\lambda^k)\\[2mm]
 		&\quad+\sigma A^{\top}
 		\big(A{x}_l^{(j+1)}-{y}_l^{(j+1)}-b+{z}_l^{(j+1)}/\sigma\big)\\[2mm]
 		&\quad+\sigma\mathcal{S}_x({x}_l^{(j+1)}-x^{(j)}),
 	\end{array}
 	\right.
 	$$
 	or, equivalently:
 	$$
 	\left\{
 	\begin{array}{rl}
 		0&\in \partial_y\mathcal{Q}({y}_l^{(j+1)})-{z}_l^{(j+1)}+\sigma\mathcal{S}_y({y}_l^{(j+1)}-{y}_l^{(j)}),\\[2mm]
 		0&=\big(b-A{x}^{(j+1)}+{y}_l^{(j+1)}\big)+A({x}_l^{(j+1)}-x^{(j)})\\[2mm]
 		&\quad+\sigma^{-1}({z}_l^{(j+1)}-{z}_l^{(j)}),\\[2mm]
 		0&\in\partial_x \mathcal{R}({x}_l^{(j+1)},\lambda^k)+A^{\top}{z}_l^{(j+1)}\\[2mm]
 		&\quad+\sigma(A^{\top}A+\mathcal{S}_x)({x}_l^{(j+1)}-x^{(j)})+A^{\top}({z}_l^{(j+1)}-{z}_l^{(j)}).
 	\end{array}
 	\right.
 	$$
 	This implies that
 	$$0\in \mathcal{\tilde{D}}(\omega_l^{j+1})+\mathcal{\tilde{S}}(\omega_l^{j+1}-\omega^{j}),$$
 	or,
 	$$
 	\omega_l^{j+1}=\mathcal{\tilde{T}}(\omega^{j}),
 	$$
 	by using the definition of $\mathcal{\tilde{T}}$.
 	
 	(b) Let $a=\mathcal{\tilde{T}}(u)$ and $b=\mathcal{\tilde{T}}(v)$. From the expression of $\mathcal{\tilde{T}}=(\mathcal{\tilde{D}}+\mathcal{\tilde{S}})^{-1} \mathcal{\tilde{S}}$, we can get that
 	$$
 	\mathcal{\tilde{S}}u=(\mathcal{\tilde{D}}+\mathcal{\tilde{S}})a, \quad \text{and} \quad \mathcal{\tilde{S}}v=(\mathcal{\tilde{D}}+\mathcal{\tilde{S}})b.
 	$$
 	By subtracting both sides of the above two equations, we get
 	$$
 	\mathcal{\tilde{S}}u-\mathcal{\tilde{S}}v=\mathcal{\tilde{D}}(a)-\mathcal{\tilde{D}}(b)+\mathcal{\tilde{S}}(a-b).
 	$$
 	Taking the inner product of both sides with $a - b$, we get
 	{\small\begin{equation}\label{lemmaproof1}
 		\big\langle a-b, \mathcal{\tilde{S}}u-\mathcal{\tilde{S}}v\big\rangle=\big\langle a-b, \mathcal{\tilde{D}}(a)-\mathcal{\tilde{D}}(b)\big\rangle+\big\langle a-b, \mathcal{\tilde{S}}(a-b)\big\rangle.
 	\end{equation}}
 	Based on Assumption \ref{ass2}, which assumes that $\mathcal{Q}(\cdot)$ and $\mathcal{R}(\cdot)$ are closed and convex, it can be concluded that the operator $\mathcal{\tilde{D}}$ is maximal monotone \cite{R2015}, and additionally, that
 	$$
 	\big\langle a-b, \mathcal{\tilde{D}}(a)-\mathcal{\tilde{D}}(b)\big\rangle\geq0.
 	$$
 	Therefore, from \eqref{lemmaproof1} , we can obtain that
 	$$
 	\big\langle a-b, \mathcal{\tilde{S}}u-\mathcal{\tilde{S}}v\big\rangle\geq\big\langle a-b, \mathcal{\tilde{S}}(a-b)\big\rangle.
 	$$
 	Since $\mathcal{S}_x$ and $\mathcal{S}_y$ are symmetric semi-positive definite and $\sigma>0$, the positive definiteness of $\mathcal{\tilde{S}}$ can be verified by computing its quadratic form. Then,
 	we can get $\tilde{\mathcal{T}}$ is firmly non-expansive with respect to $\|\cdot\|_{\mathcal{\tilde{S}}}$, that is,
 	$$
 	\|\mathcal{\tilde{T}}(u)-\mathcal{\tilde{T}}(v)\|^2_{\mathcal{\tilde{S}}}\leq\big\langle u-v, \mathcal{\tilde{T}}(u)-\mathcal{\tilde{T}}(v)\big\rangle_{\mathcal{\tilde{S}}},
 	$$
 	which implies that the claim (b) is true.
 	
 	(c) Suppose we have a sequence $(\omega^j, a^j)\in \text{gph} \ (\mathcal{\tilde{T}})$ such that
 	$(\omega^j, a^j)\rightarrow (\omega, a)$. We need to show that $(\omega, a)\in \text{gph} \ (\mathcal{\tilde{T}})$, i.e., $a=\mathcal{\tilde{T}}(\omega)$.
 	Since $(\omega^j, a^j) \in \text{gph} \ (\mathcal{\tilde{T}})$, we have:
 	$
 	\mathcal{\tilde{S}}(\omega^j - a^j) \in \mathcal{\tilde{D}}(a^j).$
 	As $j \to \infty$, $\omega^j \to \omega$ and $a^j \to a$. Since $\mathcal{\tilde{S}}$ is a linear operator, it follows from the continuity of $\mathcal{\tilde{S}}$ that,
 	$
 	\mathcal{\tilde{S}}(\omega^j - a^j) \to \mathcal{\tilde{S}}(\omega - a).
 	$
 	Also, since $\mathcal{\tilde{D}}$ is maximal monotone, it is closed. Therefore, if $a^j \to a$ and $\mathcal{\tilde{S}}(\omega^j - a^j) \to \mathcal{\tilde{S}}(\omega - a)$ with $\mathcal{\tilde{S}}(\omega^j - a^j) \in \mathcal{\tilde{D}}(a_j)$, then by closedness of $\mathcal{\tilde{D}}$, it gets that
 	$
 	\mathcal{\tilde{S}}(\omega - a) \in \mathcal{\tilde{D}}(a).
 	$
 	This is equivalent to getting
 	$
 	\mathcal{\tilde{S}} \omega \in (\mathcal{\tilde{D}} + \mathcal{\tilde{S}}) a,
 	$
 	which means that
 	$
 	a = (\mathcal{\tilde{D}} + \mathcal{\tilde{S}})^{-1} \tilde{\mathcal{S}} \omega = \mathcal{\tilde{T}}(\omega).
 	$
 	Hence, $(\omega, a) \in \text{gph}\ (\mathcal{\tilde{T}})$. The proof of the assertion (c) is completed.
 \end{proof}
 
 Under Lemma \ref{lemma1}, we further analyze the non-expansiveness and closedness of the $x$-component iterative operator.
 \begin{lemma}\label{lemma2}
 	Suppose that  Assumption \ref{ass2} holds. For any $\lambda^k\in \Lambda$, denote ${x}_l^{(j+1)}=\mathcal{T}_{\lambda^k}(x^{(j)})$  according to the $x$-subproblem \eqref{subx0}. Let $\mathcal{S}:=\sigma (A^{\top}A +\mathcal{S}_x)$.
 	Then we  have:\\
 	\begin{enumerate}
 		\item[(a)] The operator $\mathcal{T}_{\lambda^k}$ is firmly non-expansive with respect to the norm $\|\cdot\|_{\mathcal{S}}$, meaning that for any $(u, v)\in \mathbb{R}^n\times \mathbb{R}^n$, the following inequality holds
 		\begin{equation}\label{lemma21}
 			\begin{aligned}
 				&\|u - v\|^2_{\mathcal{S}} - \|\mathcal{T}_{\lambda^k}(u) - \mathcal{T}_{\lambda^k}(v)\|^2_{\mathcal{S}}\\[2mm]
 				\geq&\|(u - \mathcal{T}_{\lambda^k}(u)) - (v - \mathcal{T}_{\lambda^k}(v ))\|^2_{\mathcal{S}} \geq 0.
 			\end{aligned}
 		\end{equation}
 		\item[(b)] The operator $\mathcal{T}_{\lambda^k}$ is closed, i.e., $\text{gph}\ (\mathcal{T}_{\lambda^k})$ is closed.
 	\end{enumerate}
 \end{lemma}
 \begin{proof}
 	(a) Applying the first-order optimality conditions on \eqref{subx0}, we get
 {\small	\begin{equation}\label{lemma2proof1}
 		\begin{aligned}
 			0\in&\partial_x \mathcal{R}({x}_l^{(j+1)},\lambda^k)+A^{\top}{z}_l^{(j+1)}\\[2mm]
 			&+\sigma(A^{\top}A+\mathcal{S}_x)({x}_l^{(j+1)}-x^{(j)})+A^{\top}({z}^{(j+1)}-{z}_l^{(j)}).
 		\end{aligned}
 	\end{equation}}
 	Define
 	$\mathcal{D}_{\lambda^k}:=\partial_x \mathcal{R}(x,\lambda^k)+A^{\top}{z}_l^{(j+1)}+A^{\top}({z}_l^{(j+1)}-{z}_l^{(j)})
 	$ and $\mathcal{S}:=\sigma (A^{\top}A +\mathcal{S}_x)$, then \eqref{lemma2proof1} can be rewritten as
 	$$
 	{x}_l^{(j+1)}=\mathcal{T}_{\lambda^k}(x^{(j)}),\quad \text{with} \quad \mathcal{T}_{\lambda^k}:=(\mathcal{D}_{\lambda^k}+\mathcal{S})^{-1} \mathcal{S}.
 	$$
 	Therefore, following the the proof process of Lemma \ref{lemma1}, for all $(u, v)\in \mathbb{R}^n\times \mathbb{R}^n$, it holds that
 	$$
 	\|\mathcal{T}_{\lambda^k}(u)-\mathcal{T}_{\lambda^k}(v)\|^2_{\mathcal{S}}\leq\langle u-v, \mathcal{T}_{\lambda^k}(u)-\mathcal{T}_{\lambda^k}(v)\rangle_{\mathcal{S}}.
 	$$
 	this is to say, $\mathcal{T}_{\lambda^k}$ is  firmly non-expansive with respect to  norm $\|\cdot\|_{\mathcal{S}}$.
 	From  \cite{BC2011}, it is known that $\mathcal{T}_{\lambda^k}$ is firmly non-expansive if and only if $\mathcal{I}-\mathcal{T}_{\lambda^k}$ is also firmly non-expansive. Therefore, we get that
 	\begin{align*}
 		&\|(u -\mathcal{T}_{\lambda^k}(u)) - (v - \mathcal{T}_{\lambda^k}(v ))\|^2_{\mathcal{S}} \\[2mm]
 		\leq&	\langle u-v, (u - \mathcal{T}_{\lambda^k}(u)) - (v - \mathcal{T}_{\lambda^k}(v ))\rangle_{\mathcal{S}}\\[2mm]
 		=&\|u - v\|^2_{\mathcal{S}}-\langle u-v, \mathcal{T}_{\lambda^k}(u)-\mathcal{T}_{\lambda^k}(v)\rangle_{\mathcal{S}}\\[2mm]
 		\leq& \|u - v\|^2_{\mathcal{S}}-\|\mathcal{T}_{\lambda^k}(u)-\mathcal{T}_{\lambda^k}(v)\|^2_{\mathcal{S}},
 	\end{align*}
 	which shows that conclusion (a) is true.
 	
 	(b) The proof of the closedness of operator $\mathcal{T}_{\lambda^k}$ follows from Lemma \ref{lemma1} (c), and thus, we omit it here.
 \end{proof}
 
 For the sake of convenience in the subsequent analysis, we abbreviate $\mathcal{F}(x,\lambda)$  as $\mathcal{F}(x)$ with a fixed $\lambda$, and abbreviate $\varphi(\lambda)$  as $\mathcal{F}^*$ at optimal solution $(\bar x,\bar\lambda)$ , and then introduce the following notations:
 \begin{align*}
 	\mathcal{V}&=\mathcal{V}({\lambda}):=\arg\min_{x\in\mathbb{R}^n}\Big\{\mathcal{Q}(Ax-b)+\sum^r_{i=1}\lambda_i\mathcal{R}_i(x)\Big\},\\
 	\hat{\mathcal V}&=\hat{\mathcal V}({\lambda}):=\arg\min_{x\in \mathcal{V}\cap \mathbb{X}}\Big\{\mathcal{F}(x,\lambda)\Big\}.
 \end{align*}
 Clearly, it holds that $\hat{\mathcal{V}}\subseteq\mathcal{V}$.
 
 Using these notations and noting that $\mathcal{S}=\sigma (A^{\top}A +\mathcal{S}_x)$, we can present the following lemma.
 \begin{lemma}\label{lemma3}
 	Suppose that Assumption \ref{ass2} holds. For $ \lambda^k \in \Lambda $, let the sequence $ \{ x^{(j)} \} $ be generated by Step 1 of  ADMM-BDA with $ s_j = \frac{s}{j+1} $, where $ s \in (0, 2\lambda_{\min}(\mathcal{S})/L_\delta) $ with $ \lambda_{\min}(\mathcal{S}) $ denoting the minimum eigenvalue of  $ \mathcal{S} $. Then, for any $\bar{x} \in\mathcal{V}$, we have:\\
 	\begin{enumerate}
 		\item[(a)] $\|x_l^{(j+1)} - \bar{x}\|_{\mathcal{S}} \leq \|x^{(j)} - \bar{x}\|_{\mathcal{S}}$.\\
 		\item[(b)] The operator $\mathcal{I} - s_j \mathcal{S}^{-1} \nabla_x \mathcal{F}$ is non-expansive with respect to $\|\cdot\|_{\mathcal{S}}$. \\
 		\item[(c)] The sequences $\{x^{(j)}\}$, $\{x_l^{(j)}\}$, and $\{x_u^{(j)}\}$ are all bounded.
 	\end{enumerate}
 \end{lemma}
 \begin{proof}
 	(a) From \ref{lemma2} (a), it follows that ${x}_l^{(j+1)}=\mathcal{T}_{\lambda^k}(x^{(j)})$, where $\mathcal{T}_{\lambda^k}$ is firmly non-expansive with respect to the norm $\|\cdot\|_{\mathcal{S}}$.
 	Therefore, it can be concluded that
 	$$
 	\|x_l^{(j+1)} - \bar{x}\|_{\mathcal{S}} \leq \|x^{(j)} - \bar{x}\|_{\mathcal{S}}.
 	$$
 	(b) For any $u, v\in\mathbb{R}^n$, we have the following inequality
 	$$
 	\begin{aligned}
 		&\langle \mathcal{S}^{-1} \nabla \mathcal{F}(u)-\mathcal{S}^{-1} \nabla \mathcal{F}(v), u-v\rangle_{\mathcal{S}}\\
 		=&\langle\nabla \mathcal{F}(u)-\nabla \mathcal{F}(v), u-v\rangle \\
 		\geq & \frac{1}{{L_\delta}}\|\nabla \mathcal{F}(u)-\nabla \mathcal{F}(v)\|^2\\
 		\geq& \frac{\lambda_{\min }(\mathcal{S})}{{L_\delta}}\|\mathcal{S}^{-1} \nabla \mathcal{F}(u)-\mathcal{S}^{-1} \nabla\mathcal{F}(v)\|_{\mathcal{S}}^2,
 	\end{aligned}
 	$$
 	the first inequality follows the fact that $\mathcal{F}$ is $L_{\delta}$-smooth in Assumption \ref{ass1}.
 	Based on this inequality, it follows that
 	\begin{align*}
 		&\|(u - v) - s_j \mathcal{S}^{-1} (\nabla \mathcal{F}(u) - \nabla \mathcal{F}(v))\|_{\mathcal{S}}^2\\[2mm]
 		= &\|u - v\|_{\mathcal{S}}^2 - 2s_j \langle \mathcal{S}^{-1} (\nabla \mathcal{F}(u) - \nabla \mathcal{F}(v)), u - v \rangle_{\mathcal{S}}\\
 		& + s_j^2 \|\mathcal{S}^{-1} (\nabla \mathcal{F}(u) - \nabla \mathcal{F}(v))\|_{\mathcal{S}}^2\\
 		\leq& \|u - v\|_{\mathcal{S}}^2 - 2s_j\cdot \frac{\lambda_{\min}(\mathcal{S})}{{L_\delta}} \|\mathcal{S}^{-1} (\nabla \mathcal{F}(u) - \nabla \mathcal{F}(v))\|_{\mathcal{S}}^2 \\
 		&+ s_j^2 \|\mathcal{S}^{-1} (\nabla \mathcal{F}(u) - \nabla \mathcal{F}(v))\|_{\mathcal{S}}^2\\
 		=& \|u - v\|_{\mathcal{S}}^2 - \left( 2s_j \cdot \frac{\lambda_{\min}(\mathcal{S})}{{L_\delta}} - s_j^2 \right) \|\mathcal{S}^{-1} (\nabla \mathcal{F}(u) - \nabla \mathcal{F}(v))\|_{\mathcal{S}}^2.
 	\end{align*}
 	Note that $0<s_j \leq  \frac{2\lambda_{\min}(\mathcal{S})}{L_\delta}$, which implies
 	$     \frac{2s_j\lambda_{\min}(\mathcal{S})}{L_\delta} - s_j^2 \geq 0.$ Therefore, it can be deduced that
 	$$
 	\|(u - v) - s_j \mathcal{S}^{-1} (\nabla \mathcal{F}(u) - \nabla \mathcal{F}(v))\|_{\mathcal{S}}^2\leq\|u - v\|_{\mathcal{S}}^2,
 	$$
 	which implies that the operator $\mathcal{I} - s_j \mathcal{S}^{-1} \nabla_x \mathcal{F}$ is non-expansive with respect to $\|\cdot\|_{\mathcal{S}}$.\\

 	(c) According to Steps $1.3$ of   ADMM-BDA, we see that $x^{(j)}$ is a projection onto the convex  compact set $\mathbb{X}$, hence the sequence $\{x^{(j)}\}$ is bounded. Furthermore, from the fact of Lemma \ref{lemma3} (a), it follows that the sequence  $\{x^{(j)}_l\}$ is also bounded.
 	From the result of Lemma \ref{lemma3} (b), and from the relation
 	$x^{(j+1)}_{u} = x^{(j)} - s_j \mathcal{S}^{-1} \nabla \mathcal{F}(x^{(j)})$,
 	we have that
 	\begin{align*}
 		&\|x^{(j+1)}_{u} - \bar{x} - s_j \mathcal{S}^{-1} \nabla \mathcal{F}(\bar{x})\|_{\mathcal{S}} \\[2mm]
 		=& \|x^{(j)} - s_j \mathcal{S}^{-1} \nabla \mathcal{F}(x^{(j)}) - \bar{x} - s_j \mathcal{S}^{-1} \nabla \mathcal{F}(\bar{x})\|_{\mathcal{S}}\\[2mm]
 		\leq & \|x^{(j)}-\bar{x}\|_{\mathcal{S}}.
 	\end{align*}
 	Because $0 < s_j \leq \frac{2\lambda_{\min}(\mathcal{S})}{{L_\delta}}$, it follows that the sequence $\{x_u^{(j+1)}\}$ is also bounded.
 \end{proof}

 \begin{lemma}\label{lemma4}
 	For any given $\lambda^k \in \Lambda$, let $\{x^{(j)}\}$ be the sequence generated by step $1.3$ in Algorithm ADMM-BDA, where $s_j\in(0, \frac{2\lambda_{\min}(\mathcal{S})}{{L_\delta}})$. For any $x \in \mathcal{V}$, we have
 	\begin{equation}\label{le4}
 		\begin{aligned}
 			\mu s_j \mathcal{F}(x) \geq &\mu s_j \mathcal{F}(x_u^{(j+1)}) - \frac{1}{2}\|x - x^{(j)}\|^2_{\mathcal{S}} + \frac{1}{2}\|x - x^{(j+1)}\|^2_{\mathcal{S}} \\[2mm]
 			&+ \frac{1}{2}\|((1-\mu)x_l^{(j+1)} + \mu x_u^{(j+1)}) - x^{(j+1)}\|^2_{\mathcal{S}} \\[2mm]
 			&+ \frac{(1 - \mu)}{2}\|x^{(j)} - x_l^{(j+1)}\|_{\mathcal{S}}^2\\[2mm]
 			&+ \frac{\mu}{2}\Big(1 - s_j\frac{{L_\delta}}{\lambda_{\min}(\mathcal{S})}\Big)\|x^{(j)}- x^{(j+1)}_u\|^2_{\mathcal{S}}.
 		\end{aligned}
 	\end{equation}
 \end{lemma}
 \begin{proof}
 	We  obtain from the definition of $x^{(j+1)}_{u}$ that
 	$$
 	x^{(j+1)}_{u} = x^{(j)}- s_j  \mathcal{S}^{-1} \nabla_x \mathcal{F}(x^{(j)}, \lambda),
 	$$
 	which implies
 	$$
 	\mathcal{S}(x^{(j+1)}_u - x^{(j)}) + s_j \nabla \mathcal{F}(x^{(j)}) = 0.
 	$$
 	Since $\mathcal{S}$ is a self-adjoint linear operator, for any $x$, it holds that
 	$$
 	\langle x^{(j+1)}_u - x^{(j)}, \mathcal{S}(x - x_u^{(j+1)}) \rangle + s_j \langle \nabla \mathcal{F}(x^{(j)}), x - x_u^{(j+1)} \rangle = 0.
 	$$
 	Since $ \mathcal{F}(\cdot)$ is convex and $ \nabla \mathcal{F}$ is $L_{\delta}$-Lipschitz continuous, we have
 	\begin{equation}\label{lemma4:eq2}
 		\begin{aligned}
 			& \langle \nabla \mathcal{F}(x^{(j)}), x - x_u^{(j+1)} \rangle\\[2mm]
 			=& \langle \nabla \mathcal{F}(x^{(j)}), x - x^{(j)} \rangle + \big\langle \nabla \mathcal{F}(x^{(j)}), x^{(j)} - x_u^{(j+1)} \big\rangle\\[2mm]
 			\leq & \mathcal{F}(x) - \mathcal{F}(x^{(j)}) + \mathcal{F}(x^{(j)}) - \mathcal{F}(x_u^{(j+1)}) \\[2mm]
 			&+ \frac{{L_\delta}}{2} \|x^{(j)} - x^{(j+1)}_u\|^2_2\\[2mm]
 			= & \mathcal{F}(x) - \mathcal{F}(x_u^{(j+1)}) + \frac{{L_\delta}}{2} \|x^{(j)} - x^{(j+1)}_u\|^2_2 \\[2mm]
 			\leq & \mathcal{F}(x) - \mathcal{F}(x_u^{(j+1)}) + \frac{{L_\delta}}{2 \lambda_{\min}(\mathcal{S})} \|x^{(j)}- x^{(j+1)}_u\|^2_{\mathcal{S}}.
 		\end{aligned}
 	\end{equation}
 	Combining the expression
 	$
 	\langle x_u^{(j+1)}-x^{(j)},\mathcal{S} (x-x^{(j+1)}_u)\rangle=\frac{1}{2}(\|x-x^{(j)}\|^2_{\mathcal{S}}-\|x-x_u^{(j+1)}\|^2_{\mathcal{S}}-\|x^{(j)}-x^{(j+1)}_u\|^2_{\mathcal{S}})
 	$
 	with the inequalities in \eqref{lemma4:eq2}, it follows that for any $x$ and $\mu\in(0,1)$, we have
 	$$
 	\begin{aligned}
 		&\frac{\mu}{2}\Big(\|x-x^{(j)}\|^2_{\mathcal{S}}-\|x-x_u^{(j+1)}\|^2_{\mathcal{S}}-\|x^{(j)}-x^{(j+1)}_u\|^2_{\mathcal{S}}\Big)\\[2mm]
 		=&\mu\langle x^{(j+1)}_u-x^{(j)}, \mathcal{S}(x-x_u^{(j+1)})\rangle\\[2mm]
 		=&-\mu s_j\langle\nabla \mathcal{F}(x^{(j)}),x-x_u^{(j+1)}\rangle\\[2mm]
 		\geq&- \mu s_j\Big(\mathcal{F}(x) - \mathcal{F}(x_u^{(j+1)}) + \frac{{L_\delta}}{2 \lambda_{\min}(\mathcal{S})} \|x^{(j)}- x^{(j+1)}_u\|^2_{\mathcal{S}}\Big).
 	\end{aligned}
 	$$
 	that is,
 	\begin{equation}\label{23}
 		\begin{aligned}
 			\mu s_j \mathcal{F}(x) \geq &\mu s_j \mathcal{F}(x_u^{(j+1)}) - \frac{\mu}{2} \|x - x^{(j)}\|_{\mathcal{S}}^2 \\
 			&+ \frac{\mu}{2} \|x - x_u^{(j+1)}\|_{\mathcal{S}}^2 \\
 			&+ \frac{\mu}{2} \Big( 1 - s_j \frac{{L_\delta}}{\lambda_{\min}(\mathcal{S})} \Big) \|x^{(j)} - x_u^{(j+1)}\|_{\mathcal{S}}^2.
 		\end{aligned}
 	\end{equation}
 	From ${x}_l^{(j+1)} = \mathcal{T}_{\lambda^k}(x^{(j)})$, and given that $\mathcal{T}_{\lambda^k}$ satisfies inequality \eqref{lemma21} in Lemma \ref{lemma2}, for any $x \in \mathcal{V}$ and $\frac{1-\mu}{2} \in (0, 1)$, we obtain
 {\small	\begin{equation}\label{24}
 		\frac{1-\mu}{2}\|x^{(j)} - x\|^2_{\mathcal{S}} \geq  \frac{1-\mu}{2}\|x^{(j)} - x_l^{(j+1)}\|^2_{\mathcal{S}} + \frac{1-\mu}{2}\|x_l^{(j+1)} - x\|^2_{\mathcal{S}}.
 	\end{equation}}
 	Summing both sides of \eqref{23} and \eqref{24}, for any $x \in\mathcal{V}$, we obtain
 	\begin{equation}\label{sum}
 		\begin{aligned}
 			\mu s_j \mathcal{F}(x) \geq & \ \mu s_j \mathcal{F}(x_u^{(j+1)}) - \frac{1}{2} \|x - x^{(j)}\|_{\mathcal{S}}^2 \\[2mm]
 			&+ \frac{(1 - \mu)}{2} \|x^{(j)} - x_l^{(j+1)}\|_{\mathcal{S}}^2 \\[2mm]
 			&+ \frac{1}{2} \Big( \mu \|x - x_u^{(j+1)}\|_{\mathcal{S}}^2 + (1 - \mu) \|x^{(j+1)}_l - x\|_{\mathcal{S}}^2 \Big) \\[2mm]
 			&+ \frac{\mu}{2} \Big( 1 - s_j \frac{{L_\delta}}{\lambda_{\min}(\mathcal{S})}\Big) \|x^{(j)} - x^{(j+1)}_u\|_{\mathcal{S}}^2.
 		\end{aligned}
 	\end{equation}
 	Using the convexity of $\|\cdot\|^2_{\mathcal{S}}$ and the  firm non-expansiveness of the projection operator \cite{BC2011}, it follows that
 	\begin{align*}
 		&\mu\|x - x_u^{(j+1)}\|^2_{\mathcal{S}} + (1 - \mu)\|x^{(j+1)}_l - x\|^2_{\mathcal{S}}\\[2mm]
 		\geq &\|x - \big((1 - \mu)x_l^{(j+1)} + \mu x_u^{(j+1)}\big)\|^2_{\mathcal{S}}\\[2mm]
 		\geq &\|x - x^{(j+1)}\|^2_{\mathcal{S}} + \|((1 - \mu)x_l^{(j+1)} + \mu x_u^{(j+1)}) - x^{(j+1)}\|^2_{\mathcal{S}}.
 	\end{align*}
 	Then, from (\ref{sum}), we obtain for any $x \in \mathcal{V}$ that
 	$$
 	\begin{aligned}
 		\mu s_j \mathcal{F}(x) \geq & \ \mu s_j \mathcal{F}(x_u^{(j+1)}) - \frac{1}{2} \|x - x^{(j)}\|^2_{\mathcal{S}} + \frac{1}{2} \|x - x^{(j+1)}\|^2_{\mathcal{S}} \\[2mm]
 		& + \frac{1}{2} \| ( (1 - \mu)x_l^{(j+1)} + \mu x_u^{(j+1)} ) - x^{(j+1)} \|^2_{\mathcal{S}} \\[2mm]
 		& + \frac{(1 - \mu)}{2} \|x^{(j)} - x_l^{{(j+1)}}\|^2_{\mathcal{S}}\\[2mm]
 		& + \frac{\mu}{2} \Big( 1 - s_j \frac{{L_\delta}}{\lambda_{\min}(\mathcal{S})} \Big) \|x^{(j)} - x^{(j+1)}_u\|^2_{\mathcal{S}},
 	\end{aligned}
 	$$
 	which completes the proof.
 \end{proof}
 
 \begin{proposition}\label{pro1}
 	For any given $\lambda^k \in \Lambda$  and $\mu \in (0,1)$, let $\{x^{(j)}\}$ be the sequence generated in Step $1$ of the Algorithm ADMM-BDA, where $s_j\in (0, \frac{\lambda_{\min}(\mathcal{S})}{{L_\delta}})$ with $\lim_{j \to \infty}s_j=0$ and  $\sum_{j=0}^{\infty}s_j= +\infty$ hold. Assuming that $\hat{\mathcal{V}}$ is nonempty, we have
 	$$
 	\lim_{j \to \infty}\mathrm{dist}_{\mathcal{S}}(x^{(j)}, \hat{\mathcal{V}}) = 0, \quad \text{and} \quad
 	\lim_{j \to \infty} \mathcal{F}(x^{(j)}, \lambda^k) = \varphi(\lambda^k).
 	$$
 \end{proposition}
 \begin{proof}
 	Let $\gamma>0$ be a parameter such that $\gamma<\min\Big\{\frac{(1 - \mu)}{2}, \frac{\mu}{2}\Big(1 - \frac{s_jL_\delta}{\lambda_{\min}(\mathcal{S})}\Big)\Big\}$. We define a sequence $\{\nu_k\}$ as follows
 	\begin{align*}
 		\nu_{\kappa}\overset{\text{def}}{=} \max_j\Big\{& j\in\mathbb{N}\mid j\leq \kappa, \ \text{and} \ \gamma\|x^{(j-1)}-x_u^{(j)}\|_{\mathcal{S}}^2\\[2mm]
 		&+\gamma\|x^{(j-1)}-x_l^{(j)}\|_{\mathcal{S}}^2\\[2mm]
 		&+\frac{1}{4}\| ((1-\mu)x_l^{(j)}+\mu x_u^{(j)})-x^{(j)}\|_{\mathcal{S}}^2\\[2mm]
 		&+{\mu s_{j-1}}(\mathcal{F}(x_u^{(j)})-\mathcal{F}^*)<0\Big\},
 	\end{align*}
 	where $\kappa$ is a positive integer.
 	Similar to the work of Liu et al. \cite{LLZ2022} and Cabot \cite{C2005}, we consider the following two cases.
 	
 	{Case (a):} Suppose that the sequence $\{\nu_\kappa\}$ is finite, i.e., there exists $j_0\in\mathbb{N}$ such that for all $j\geq j_0$, we have
 {\small	\begin{align*}			
 		&\gamma\|x^{(j-1)}-x_l^{(j)}\|_{\mathcal{S}}^2
 		+\gamma\|x^{(j-1)}-x_u^{(j)}\|_{\mathcal{S}}^2\\[2mm]
 		&+\frac{1}{4}\| ((1-\mu)x_l^{(j)}+\mu x_u^{(j)})-x^{(j)}\|_{\mathcal{S}}^2+{\mu s_{j-1}}(\mathcal{F}(x_u^{(j)})-\mathcal{F}^*)\geq0.
 	\end{align*}}
 	Let $\bar{x}$ be any point in $\hat{\mathcal{V}}\subseteq\mathcal{ V}$, and substitute the $x$ in \eqref{le4} of Lemma \ref{lemma4}  with $\bar{x}$. Notice that $\mu \in (0,1)$, we have
 	\begin{equation}\label{casea}
 		\begin{aligned}
 			&\frac{1}{2}\|\bar{x} - x^{(j)}\|_{\mathcal{S}}^2\\[2mm]
 			\geq & \frac{1}{2}\|\bar{x} - x^{(j+1)}\|_{\mathcal{S}}^2 + \Big(\frac{(1 - \mu)}{2} - \gamma\Big) \|x^{(j)} - x_l^{(j+1)}\|_{\mathcal{S}}^2 \\[2mm]
 			& + \Big(\frac{\mu}{2}(1 - \frac{s_jL_\delta}{\lambda_{\min}(\mathcal{S})}) - \gamma\Big) \|x^{(j)} - x_u^{(j+1)}\|_{\mathcal{S}}^2 \\[2mm]
 			& + \frac{1}{4} \|((1 - \mu)x_l^{(j+1)} + \mu x_u^{(j+1)}) - x^{(j+1)}\|_{\mathcal{S}}^2 \\[2mm]
 			& + \gamma \|x^{(j)} - x_u^{(j)}\|_{\mathcal{S}}^2 \\[2mm]
 			& + \frac{1}{4} \|((1 - \mu)x_l^{(j+1)} + \mu x_u^{(j+1)}) - x^{(j+1)}\|_{\mathcal{S}}^2 \\[2mm]
 			& + \gamma \|x^{(j)} - x_l^{(j)}\|_{\mathcal{S}}^2 \\[2mm]
 			& + \mu s_j (\mathcal{F}(x_u^{(j+1)}) - \mathcal{F}^*).
 		\end{aligned}
 	\end{equation}
 	For any $j \geq j_0$, it follows from\eqref{casea} and \cite[Lemma A.4]{LLZ2022} that $\lim_{j \to \infty} \|\bar{x} - x^{(j)}\|_{\mathcal{S}_x}^2$ exists, and that
 {\small	\begin{align*}
 		\sum_{j=0}^{\infty}&\Big[\Big(\frac{(1 - \mu)}{2} - \gamma\Big) \|x^{(j)} - x_l^{(j+1)}\|_{\mathcal{S}}^2\\
 		& + \frac{\mu}{2}\Big((1 - \frac{s_jL_\delta}{\lambda_{\min}(\mathcal{S})} ) - \gamma\Big) \|x^{(j)} - x_u^{(j+1)}\|_{\mathcal{S}}^2 \\
 		& + \frac{1}{4} \big\|\big((1 - \mu)x_l^{(j+1)} + \mu x_u^{(j+1)}\big) - x^{(j+1)}\big\|_{\mathcal{S}}^2 \\[2mm]
 		& + \gamma \|x^{(j)} - x_u^{(j)}\|_{\mathcal{S}}^2 \\
 		& + \frac{1}{4} \big\|\big((1 - \mu)x_l^{j+1} + \mu x_u^{j+1}\big) - x^{j+1}\big\|_{\mathcal{S}}^2 + \gamma \|x^{(j)} - x_l^{(j)}\|_{\mathcal{S}}^2 \\
 		& + \mu s_j (\mathcal{F}(x_u^{(j+1)}) - \mathcal{F}^*)\Big]<\infty.
 	\end{align*}}
 	Combining with the fact $\gamma<\min\big\{\frac{(1 - \mu)}{2}, \frac{\mu}{2}(1 - \frac{sL_\delta}{\lambda_{\min}(\mathcal{S})})\big\}$, we   get
 	$$
 	\begin{aligned}
 		&\sum_{j=0}^{\infty}\|x^{(j)}-x_u^{(j+1)}\|_{\mathcal{S}}^2<\infty,\quad
 		\sum_{j=0}^{\infty}\|x^{(j)}-x_l^{(j+1)}\|_{\mathcal{S}}^2<\infty,\\
 		&\sum_{j=0}^{\infty}\|((1-\mu)x_l^{(j+1)}+\mu x_u^{(j+1)})-x^{(j+1)}\|_{\mathcal{S}}^2<\infty,\\
 		&\sum_{j=0}^{\infty}s_j(\mathcal{F}(x_u^{(j+1)}-\mathcal{F}^*))<\infty.
 	\end{aligned}
 	$$
 	
 	With any fixed $\lambda^k\in\Lambda$, we now prove the existence of a sub-sequence $\{x^{(\tilde{j})}\} \subseteq \{x^{(j)}\}$ such that $\lim_{\tilde{j} \to \infty} \mathcal{F}(x^{(\tilde{j})}) \leq \mathcal{F}^*$. This clearly holds for every $\hat{j} > 0$, there exists some $j > \hat{j}$ such that $\mathcal{F}(x^{(j)}) \leq \mathcal{F}^*$.
 	Therefore, we only need to consider the case where there exists some $\hat{j} > 0$ such that $\mathcal{F}(x^{(j)}) > \mathcal{F}^*$ for all $j > \hat{j}$. If there does not exist a sequence $\{x^{(\tilde{j})}\} \subseteq \{x^{(j)}\}$ such that $\lim_{\tilde{j} \to \infty} \mathcal{F}(x^{(\tilde{j})}) \leq \mathcal{F}^*$, then there must exist a $\epsilon > 0$ and a $j_1 \geq \max\{\hat{j}, j_0\}$ such that $\mathcal{F}(x^{(j)}) - \mathcal{F}^* \geq 2\epsilon$ for all $j \geq j_1$.
 	
 	Since $\mathbb{X}$ is a convex compact set, it follows from Lemma \ref{lemma3} that the sequences $\{x^{(j)}\}$ and $\{x^{(j+1)}_u\}$ are bounded. Moreover, since $\mathcal{F}$ is continuous and $\mathcal{S}\succ 0$, we get $\lim_{j \to \infty} \|x^{(j)} - x_u^{(j+1)}\|_{\mathcal{S}}^2 = 0$. Thus, there exists $j_2 \geq j_1$ such that $|\mathcal{F}(x^{(j)}) - \mathcal{F}(x_u^{(j+1)})| < \epsilon$ for all $j \geq j_2$. Consequently, $\mathcal{F}(x_u^{(j+1)}) - \mathcal{F}^* \geq \epsilon$ for all $j \geq j_2$. Then we have
 	$$
 	\epsilon \sum_{j=j_2}^{\infty} s_j \leq \sum_{j=j_2}^{\infty} s_j (\mathcal{F}(x_u^{(j+1)}) - \mathcal{F}^*) < \infty,
 	$$
 	where the last inequality follows from $\sum_{j=0}^{\infty} (\mathcal{F}(x_u^{(j+1)}) - \mathcal{F}^*) < \infty$.
 	This result contradicts the fact that $\sum_{j=0}^{\infty}s_j= +\infty$. Hence, we have   proved the existence of a sequence ${x^{\tilde{j}}} \subseteq {x^j}$ such that $\lim_{\tilde{j} \to \infty} \mathcal{F}(x^{(\tilde{j})}) \leq \mathcal{F}^*$.
 	
 	Since the sequences $\{x^{(\tilde{j})}\}$ and $\{x^{(j+1)}_u\}$ are bounded and $\sum_{j=0}^{\infty}\|x^{(j)}-x_l^{(j+1)}\|_{\mathcal{S}}^2<\infty$, we can take a subsequence such that $\lim_{\tilde{j} \to \infty} x_l^{(\tilde{j}+1)} = \lim_{\tilde{j} \to \infty}x^{(\tilde{j})} = \tilde{x}$. By the continuity of $\mathcal{F}$, we have $\mathcal{F}(\tilde{x}) = \lim_{\tilde{j} \to \infty} \mathcal{F}(x^{(\tilde{j})}) \leq \mathcal{F}^*$. Furthermore, since $x_l^{(\tilde{j}+1)} = \mathcal{T}_{\lambda^k}(x^{(\tilde{j})})$, taking the limit as $\tilde{j} \to \infty$ and by the closedness of $\mathcal{T}_{\lambda^k}$ from Lemma \ref{lemma2}, we obtain
 	$$
 	\tilde{x} = \mathcal{T}_{\lambda^k}(\tilde{x}),
 	$$
 	hence $\tilde{x} \in \mathcal{V}$.
 	Combining with $\mathcal{F}(\tilde{x}) \leq \mathcal{F}^*$, we can get $\tilde{x} \in \hat{\mathcal{V}}$. Then, by
 	taking $\bar{x} = \tilde{x}$ and since $\lim_{j \to \infty} \|\bar{x} - x^{(j)}\|_{\mathcal{S}}^2$ exists,
 	we obtain $\lim_{j \to \infty} \|\bar{x} - x^{(j)}\|_{\mathcal{S}}^2 = 0$ due to $\mathcal{S} \succ 0$.
 	Consequently, it holds that
 	$$
 	\lim_{j \to \infty} \mathrm{dist}_{\mathcal{S}}(x^{(j)}, \hat{\mathcal{V}}) = 0, \quad \text{and} \quad
 	\lim_{j \to \infty} \mathcal{F}(x^{(j)}, \lambda^k) = \varphi(\lambda^k).
 	$$
 	This completes the proof in this case.
 	
 	{{Case (b):}} Suppose that the sequence $\{\nu_\kappa\}$ is infinite, i.e., for every $j_0\in\mathbb{N}$, there exists $j\geq j_0$ such that:
 	\begin{align*}
 		&\gamma\|x^{(j-1)}-x_l^{(j)}\|_{\mathcal{S}}^2\\
 		&+\gamma\|x^{(j-1)}-x_u^{(j)}\|_{\mathcal{S}}^2+\frac{1}{4}\| ((1-\mu)x_l^{(j)}&+\mu x_u^{(j)})-x^{(j)}\|_{\mathcal{S}}^2\\[2mm]
 		&+\mu s_{j-1}(\mathcal{F}(x_u^{(j)})-\mathcal{F}^*)<0.
 	\end{align*}
 	Let $x=\Pi_{\hat{\mathcal{V}}}^{\mathcal{S}}(x^{(j)})\subset \hat{\mathcal{V}}$, and take it into the inequality \eqref{le4} of Lemma \ref{lemma4}, we can get
 	$$
 	\begin{aligned}
 		&\frac{1}{2}\mathrm{dist}^2_{\mathcal{S}}(x^{(j)}, \hat{\mathcal{V}})\\[2mm]
 		\geq&\frac{1}{2}\mathrm{dist}^2_{\mathcal{S}}(x^{(j+1)}, \hat{\mathcal{V}})
 		+(\frac{(1 -\mu)}{2}-\gamma)\|x^{(j)} - x_l^{(j+1)}\|_{\mathcal{S}}^2\\[2mm]		
 		&+\Big(\frac{\mu}{2}(1-s_j\frac{L_\delta}{\lambda_{\min}({\mathcal{S}})})-\gamma\Big)\|x^{(j)}-x_u^{(j+1)}\|_{\mathcal{S}}^2\\[2mm]
 		&+\frac{1}{4}\Big\|((1-\mu)x_l^{(j+1)}+\mu x_u^{(j+1)})-x^{(j+1)}\|_{\mathcal{S}}^2\\[2mm]
 		&+\gamma\|x^{(j)}-x_u^{(j+1)}\Big\|_{\mathcal{S}}^2\\[2mm]
 		&+\frac{1}{4}\Big\|((1-\mu)x_l^{(j+1)}+\mu x_u^{(j+1)})-x^{(j+1)}\|_{\mathcal{S}}^2\\[2mm]
 		&+\gamma\|x^{(j)}-x_l^{(j+1)}\Big\|_{\mathcal{S}}^2\\[2mm]
 		&+\mu s_j\Big(\mathcal{F}(x_u^{(j+1)})-\mathcal{F}^*\Big).
 	\end{aligned}
 	$$
 	When $\nu_\kappa<\kappa $, we have
 	\begin{align*}			
 		&\gamma\|x^{(j)}-x_l^{(j+1)}\|_{\mathcal{S}}^2\\[2mm]
 		&+\gamma\|x^{(j)}-x_u^{(j+1)}\|_{\mathcal{S}}^2\\[2mm]
 		&+\frac{1}{4}\| ((1-\mu)x_l^{(j+1)}+\mu x_u^{(j+1)})-x^{(j+1)}\|_{\mathcal{S}}^2\\[2mm]
 		&+\mu s_j(\mathcal{F}(x_u^{(j+1)})-\mathcal{F}^*)\geq0
 	\end{align*}
 	holds for any $\nu_\kappa \leq j \leq \kappa-1$. Thus, for any $\nu_\kappa \leq j \leq \kappa-1$, we can obtain
 	$$\frac{1}{2}\mathrm{dist}^2_{\mathcal{S}}(x^{(j+1)}, \hat{\mathcal{V}})-\frac{1}{2}\mathrm{dist}^2_{\mathcal{S}}(x^{(j)}, \hat{\mathcal{V}})\leq0.
 	$$
 	Summing up the inequalities from $\nu_\kappa$ to $\kappa$, we obtain
 	\begin{equation}\label{21}
 		\frac{1}{2}\mathrm{dist}^2_{\mathcal{S}}(x^{(\kappa)}, \hat{\mathcal{V}})-\frac{1}{2}\mathrm{dist}^2_{\mathcal{S}}(x^{(\nu_\kappa)}, \hat{\mathcal{V}})\leq0.
 	\end{equation}
 	When $\nu_\kappa=\kappa$, the inequality \eqref{21} also  holds for $\mathrm{dist}_{\mathcal{S}}(x^{{(\kappa)}}, \hat{\mathcal{V}})=\mathrm{dist}_{\mathcal{S}}(x^{(\nu_\kappa)},\hat{\mathcal{V}})$. Thus, if we want to obtain $\lim_{k\to \infty}\mathrm{dist}_{\mathcal{S}}(x^{(\kappa)}, \hat{\mathcal{V}}) = 0$, it can be seen from inequality \eqref{21} that we only need to prove $\lim_{\kappa\to \infty}\mathrm{dist}_{\mathcal{S}}(x^{(\nu_\kappa)}, \hat{\mathcal{V}})= 0$.
 	
 	According to the definition of $\nu_\kappa$, for all $j \in \{\nu_\kappa\}$ with different $\kappa$, we have $\mathcal{F}^* >\mathcal{F}(x_u^{(j)})$.
 	Since $\mathbb{X}$ is a convex compact set, it follows from Lemma \ref{lemma3} that both $\{x^{(\nu_\kappa)}\}$ and $\{x_u^{(\nu_\kappa)}\}$
 	are all bounded, and therefore the set $\{\mathrm{dist}_{\mathcal{S}}(x^{(\nu_\kappa)}, \hat{\mathcal{V}})\}$ is also bounded.
 	Assuming that $\mathcal{F}$ has a lower bound $F_0$, we have
 	$$
 	0 \leq \mathcal{{F}}^* - \mathcal{F}(x_u^{(j)}) \leq \mathcal{F}^* - F_0.
 	$$
 	From the definition of $\nu_\kappa$, we  obtain
 	\begin{align*}
 		&\gamma\|x^{(j-1)}-x_u^{(j)}\|^2_{\mathcal{S}}\\[2mm]
 		&+\gamma\|x^{(j-1)}-x_l^{(j)}\|^2_{\mathcal{S}}
 		+\frac{1}{4}\| ((1-\mu)x_l^{(j)}+\mu x_u^{(j)})-x^{(j)}\|^2_{\mathcal{S}}\\[2mm]
 		<& \mu s_j(\mathcal{F}^*-\mathcal{F}(x_u^{(j)})) \leq \mu s_j(\mathcal{F}^*-F_0).
 	\end{align*}
 	Furthermore, since $\lim_{\kappa\to\infty}\nu_\kappa = +\infty$ and $\lim_{j \to \infty}s_j=0$, we can obtain
 	$$
 	\lim_{\kappa\to \infty}\|x^{(\nu_{\kappa-1})} - x_u^{(\nu_\kappa)}\|_{\mathcal{S}} = 0,
 	\quad
 	\lim_{\kappa\to \infty}\|x^{(\nu_{\kappa-1})} - x_l^{(\nu_\kappa)}\|_{\mathcal{S}} = 0,
 	$$
 	and
 	$$
 	\lim_{\kappa\to\infty}\|((1-\mu)x_l^{(\nu_\kappa)} + \mu x_u^{(\nu_\kappa)}) - x^{(\nu_\kappa)}\|_{\mathcal{S}} = 0.
 	$$
 	Let $\tilde{x}$ be an arbitrary limit point of $\{x^{(\nu_\kappa)}\}$, and $\{x^{(\hat{j})}\}$ be a subsequence of $\{x^{(\nu_\kappa)}\}$ such that $\lim_{\hat{j}\to\infty}x^{(\hat{j})}=\tilde{x}$. From the following inequalities
 	\begin{align*}
 		&\lim_{\kappa\to\infty}\|x^{(\nu_{\kappa-1})} - x^{(\nu_\kappa)}\|_{\mathcal{S}}\\[2mm]
 		=&\lim_{\kappa\to\infty}\|x^{(\nu_\kappa - 1)} - ((1-\mu)x_l^{(\nu_\kappa)} + \mu x_u^{(\nu_\kappa)}) \\[2mm]
 		&+ ((1-\mu)x_l^{(\nu_\kappa)} + \mu x_u^{(\nu_\kappa)}) - x^{(\nu_\kappa)}\|_{\mathcal{S}}\\[2mm]
 		\leq&\lim_{\kappa\to\infty}\|x^{(\nu_\kappa - 1)} - ((1-\mu)x_l^{(\nu_\kappa)} + \mu x_u^{(\nu_\kappa)})\|_{\mathcal{S}}\\[2mm]
 		& + \lim_{n\to\infty}\|((1-\mu)x_l^{(\nu_\kappa)} + \mu x_u^{(\nu_\kappa)}) - x^{(\nu_\kappa)}\|_{\mathcal{S}}\\[2mm]
 		\leq&(1-\mu)\lim_{\kappa\to\infty}\|x^{(\nu_\kappa - 1)}-x_l^{(\nu_\kappa)}\|_{\mathcal{S}}\\[2mm]
 		&+\mu\lim_{\kappa\to\infty}\|x^{(\nu_\kappa - 1)}- x_u^{(\nu_\kappa)})\|_{\mathcal{S}}\\[2mm]
 		&+ \lim_{\kappa\to\infty}\|((1-\mu)x_l^{(\nu_\kappa)} + \mu x_u^{(\nu_\kappa)}) - x^{(\nu_\kappa)}\|_{\mathcal{S}}=0.
 	\end{align*}
 	where $\mathcal{S}\succ0$, we can obtain that $\lim_{\hat{j}\to\infty}x^{(\hat{j}-1)} = \tilde{x}$.
 	Furthermore, because $\lim_{\hat{j}\to\infty}\|x^{(\hat{j}-1)} - x_l^{(\hat{j})}\|_{\mathcal{S}} = 0$, we have $\lim_{\hat{j}\to\infty}x_l^{(\hat{j})} = \tilde{x}$. Furthermore, because $x_l^{(\hat{j})} = \mathcal{T}_{\lambda^k}(x^{(\hat{j}-1)})$, taking the limit as $\hat{j} \to \infty$ and using the closedness of $\mathcal{T}_{\lambda^k}$ from Lemma \ref{lemma2}, we obtain
 	$$
 	\tilde{x} = \mathcal{T}_{\lambda^k}(\tilde{x}),
 	$$
 	hence $\tilde{x} \in \mathcal{V}$. Since for all $j \in \{\nu_\kappa\}$ with different $\kappa$, we have $\mathcal{F}^* > \mathcal{F}(x_u^{(j)})$, it follows that for all $\hat{j}$, it holds that $\mathcal{F}^* > F(x_u^{(\hat{j})})$.
 	Then, due to the continuity of $\mathcal{F}$ and the fact that $\lim_{n \to \infty} \|x_u^{(\nu_\kappa)} - x^{(\nu_\kappa)}\|_{\mathcal{S}} = 0$ with $\mathcal{S} \succ 0$, we have $\mathcal{F}^* > \mathcal{F}(\tilde{x})$, so we can get $\tilde{x} \in \hat{\mathcal{V}}$ and  $\lim_{\hat{j} \to \infty} \mathrm{dist}(x^{(\hat{j})}, \hat{\mathcal{V}}) = 0$.
 	Since we have proven that for any limit point $\tilde{x}$ of the sequence $\{x^{(\nu_\kappa)}\}$, it holds that $\tilde{x} \in \hat{\mathcal{V}}$. Then we can conclude that $\lim_{\kappa \to \infty} \mathrm{dist}_{\mathcal{S}}(x^{(\nu_\kappa)}, \hat{\mathcal{V}}) = 0$ from $\{x^{(\nu_\kappa)}\}$ and $\{\mathrm{dist}_{\mathcal{S}}(x^{(\nu_\kappa)}, \hat{\mathcal{V}})\}$ are all bounded. Therefore, it follows that
 	$$
 	\begin{aligned}
 		&\lim_{j \to \infty} \mathrm{dist}_{\mathcal{S}}(x^{(j)}, \hat{\mathcal{V}}) = 0\\[2mm]
 		&\lim_{j \to \infty} \mathcal{F}(x^{(j)}, \lambda^k) = \varphi(\lambda^k),
 	\end{aligned}
 	$$
 	this completes the proof   in the second case.
 \end{proof}
 
 Based on the above discussion, we can prove that the sequence $\{x^{(j)}\}$  generated by Step $1$ in Algorithm ADMM-BDA parameterized by $\lambda^k$ converges to the solution set of $\inf_{x\in \mathcal{V}\cap \mathbb{X}}\mathcal{F}(x, \lambda^k)$.
 Meanwhile, for any $\lambda^k\in \Lambda$, the sequence $\{x^{(j)}\}$ also converges uniformly with respect to $\|x^{(j)} - \mathcal{T}_{\lambda^k}(x^{(j)})\|^2_{\mathcal{S}}$.
 This implies that the convergence result does not depend on the parameter $\lambda$, and that, for any  $\lambda\in \Lambda$, the sequence $\{x^{(j)}(\lambda)\}$ (i.e., $x^{(j)}$ parametrized by $\lambda$) converges in the same manner.
 
 \begin{theorem}\label{th1}
 	Let $\{x^{(j)}(\lambda)\}$ be the sequence generated by Step $1$ in Algorithm ADMM-BDA parameterized by $\lambda$ with $s_j = \frac{1}{j+1}$ and $s \in \big(0, \frac{\lambda_{\min}(\mathcal{S})}{{L_\delta}}\big)$, where $\lambda_{\min}(\mathcal{S})$ denotes the smallest eigenvalue of the matrix $\mathcal{S}$. Then, we have\\
 	\begin{enumerate}
 		\item[(a)] For any $\lambda \in \Lambda$, it holds that
 		$$
 		\begin{aligned}
 			&\lim_{j \to \infty} \mathrm{dist}\big(x^{(j)}(\lambda), \mathcal{T}_{\lambda}(x^{(j)}(\lambda))\big) = 0,\\[2mm]
 			&\lim_{j \to \infty} \mathcal{F}\big(x^{(j)}(\lambda), \lambda\big) = \varphi(\lambda).
 		\end{aligned}
 		$$
 		\item[(b)] Furthermore, there exists a constant $C > 0$ such that for any $\lambda \in \Lambda$, it holds that
 		$$
 		\|x^{(j)}(\lambda) - \mathcal{T}_{\lambda}(x^{(j)}(\lambda))\|_{\mathcal{S}} \leq C \sqrt{\frac{1 + \ln(1 + j)}{j^{{1}/{4}}}}.
 		$$
 	\end{enumerate}
 \end{theorem}
 \begin{proof}
 	The claim (a) can be immediately obtained from the proof process in Proposition \ref{pro1}. Since both $\mathbb{X}$ and $\Lambda$ are compact set and $\mathcal{F}(x, \lambda)$ is continuous on $\mathbb{X} \times \Lambda$, we get that $\mathcal{F}(x, \lambda)$ is uniformly bounded on $\mathbb{X} \times \Lambda$, and that $\varphi(\lambda) = \inf_{x \in V \cap \mathbb{X}} \mathcal{F}(x, \lambda)$ is bounded on $\Lambda$.
 	Combining with the assumptions, we can conclude that $\mathcal{F}(x, \lambda)$ is bounded by $F_0$ on $\mathbb{X} \times \Lambda$. From \cite[Proposition A.2]{LLZ2022}, it follows that there exists a constant $C > 0$ such that for any $\lambda \in \Lambda$, we have
 	$$
 	\|x^{(j)}(\lambda) - \mathcal{T}_{\lambda}(x^{(j)}(\lambda))\|_{\mathcal{S}} \leq C \sqrt{\frac{1 + \ln(1 + j)}{j^{{1}/{4}}}},
 	$$
 	which shows the claim (b) holds.
 \end{proof}
 
 Based on the uniform convergence results of the sequence $\{x^{(j)}(\lambda)\}$, and inspired by the work of Liu et al. \cite{LMY2020}, we analyze the convergence of Algorithm ADMM-BDA for \eqref{model1} with respect to variables both $x$ and $\lambda$.

 \begin{proposition}\label{pro2}
 	Assume that both $\mathbb{X}$ and $\Lambda$ are convex and compact. Suppose that for any $\lambda \in \Lambda$, $\{x^{(J_k)}(\lambda)\} \subset \mathbb{X}$, and that, for any $\epsilon > 0$, there exists a $j(\epsilon) > 0$ such that when $j > j(\epsilon)$ it holds that
 	$$
 	\sup_{\lambda \in \Lambda} \|x^{(j)}(\lambda)- \mathcal{T}_{\lambda}(x^{(j)}(\lambda))\| \leq \epsilon.
 	$$
 	Meanwhile,  suppose for every $\lambda \in \Lambda$, it holds that
 	$$
 	\lim_{J_k\to \infty} \varphi_{J_k}(\lambda) \to \varphi(\lambda),
 	$$
 	where $\varphi_{J_k}(\lambda)$ and $x^{(J_k)}(\lambda)$ have been defined in \eqref{model2} with regarding to an integer $J$.
 	Let $\lambda^{k} \in \arg\min_{\lambda \in \Lambda} \{\varphi_{J_k}(\lambda)\}$. Then we have:
 	\begin{enumerate}
 		\item[(a)] As $J_k\to\infty$, any limit point $(\bar{x}, \bar{\lambda})$ of the sequence $\{x^{(J_k)}({\lambda^{k}}), \lambda^{k}\}$ satisfies $\bar{\lambda} \in \arg\min_{\lambda \in \Lambda} \varphi(\lambda)$ and $\bar{x} = \mathcal{T}_{\bar{\lambda}}(\bar{x})$.
 		\item[(b)] As $J_k \to \infty$, it also holds that $\inf_{\lambda \in \Lambda} \varphi_{J_k}(\lambda) \to \inf_{\lambda \in \Lambda} \varphi(\lambda)$.
 	\end{enumerate}
 \end{proposition}
 \begin{proof}
 	(a) Suppose  for any limit point $(\bar{x}, \bar{\lambda})$ of the sequence $\{x^{(J_k)}({\lambda^{k}}), \lambda^{k}\}$, let $\mathcal{K}:=\{1,2,\ldots\}$ and $\{(x^{(J_i)}({\lambda^{i}}), \lambda^{i})\}_{i\in\mathcal{K}}$ be a subsequence of $\{(x^{(J_k)}({\lambda^{k}}), \lambda^{k})\}$ such that $x^{(J_i)} ({\lambda^{i}}) \to \bar{x} \in \mathbb{X}$ and $\lambda^{i} \to \bar{\lambda} \in \Lambda$.
 	According to the assumption that, for any $\epsilon > 0$, there exists $j(\epsilon) > 0$ such that for any $i > j(\epsilon)$, we have
 	$$
 	\|x^{(J_i)}({\lambda^{i}}) - \mathcal{T}_{\lambda^{i}}(x^{(J_i)})\| \leq \epsilon.
 	$$
 	Since $\mathcal{T}_{\lambda^{i}}$ is closed on $\mathbb{X}$, we get
 	$$
 	\|\bar{x} - \mathcal{T}_{\bar{\lambda}}(\bar{x})\| \leq \epsilon.
 	$$
 	Because $\epsilon$ is an arbitrarily chosen, we conclude that $\bar{x} = \mathcal{T}_{\bar{\lambda}}(\bar{x})$.
 	Next, since $\mathcal{F}$ is continuous at $(\bar{x}, \bar{\lambda})$, for any $\epsilon > 0$, there exists $j(\epsilon) > 0$ such that for any $i > j(\epsilon)$, the following inequality holds:
 	$$
 	\mathcal{F}(\bar{x}, \bar{\lambda}) \leq \mathcal{F}\big(x^{(J_i)}(\lambda^{i}), \lambda^{i}\big) + \epsilon.
 	$$
 	That is, for any $i > j(\epsilon)$ and $\lambda \in \Lambda$, it holds that
 	\begin{equation}\label{pp31}
 		\begin{aligned}
 			\varphi(\bar{\lambda}) &= \inf_{x \in{\mathcal{V}} \cap \mathbb{X}} \mathcal{F}(x, \bar{\lambda}) \leq \mathcal{F}(\bar{x}, \bar{\lambda}) \leq \mathcal{F}(x^{(J_i)}_{\lambda^{i}}, \lambda^{i}) + \epsilon \\[2mm]
 			&= \varphi_{J_i}(\lambda^i) + \epsilon \leq \varphi_{J_i}(\lambda) + \epsilon.
 		\end{aligned}
 	\end{equation}
 	Letting $i \to \infty$, by assumption, we obtain for any $\lambda \in \Lambda$ that
 	$$
 	\varphi(\bar{\lambda}) \leq \lim_{i \to \infty} \varphi_{J_i}(\lambda) + \epsilon = \varphi(\lambda) + \epsilon.
 	$$
 	Taking $\epsilon \to 0$, we have
 	$$
 	\varphi(\bar{\lambda}) \leq \varphi(\lambda), \quad \forall \ \lambda \in \Lambda,
 	$$
 	which also implies $\bar{\lambda} \in \arg\min_{\lambda \in \Lambda} \varphi(\lambda)$.

 	(b) For any $\lambda \in \Lambda$, it holds that $\inf_{\lambda \in \Lambda} \varphi_{J_k}(\lambda) \leq \varphi_{J_k}(\lambda)$. Letting $J_k \to \infty$ and using the relation $\lim_{J_k \to \infty} \varphi_{J_k}(\lambda) \to \varphi(\lambda)$, we obtain
 	$$
 	\lim_{J_k \to \infty} \sup \Big\{ \inf_{\lambda \in \Lambda} \varphi_{J_k}(\lambda)\Big\} \leq \varphi(\lambda), \quad \forall \lambda \in \Lambda.
 	$$
 	Therefore,
 	$$
 	\lim_{J_k \to \infty} \sup \Big\{ \inf_{\lambda \in \Lambda} \{\varphi_{J_k}(\lambda)\} \Big\} \leq \inf_{\lambda \in \Lambda} \varphi(\lambda).
 	$$
 	Thus, as $J_k\to\infty$, if the claim $\inf_{\lambda \in \Lambda} \varphi_{J_k}(\lambda) \to \inf_{\lambda \in \Lambda} \varphi(\lambda)$ does not hold, there exists a $\delta > 0$ and a subsequence $\{\lambda^{l}\}_{l\in\mathcal{K}}$ of $\{\lambda^k\}$ such that
 	\begin{equation}\label{pp32}
 		\inf_{\lambda \in \Lambda} \varphi_{J_l}(\lambda) < \inf_{\lambda \in \Lambda} \varphi(\lambda) - \delta, \quad \forall \ l.
 	\end{equation}
 	Since $\Lambda$ is compact, without loss of generality, we consider the subsequence $\lambda^l \to \bar{\lambda} \in \Lambda$.
 	Moreover, by \eqref{pp31}, for any $\epsilon > 0$, there exists $j(\epsilon) > 0$ such that for all $l > j(\epsilon)$, it holds that
 	$$
 	\varphi(\bar{\lambda}) \leq \varphi_{J_l}(\lambda^l) + \epsilon.
 	$$
 	Letting $l \to \infty$ and $\epsilon \to 0$, and by the definition of $\lambda^l$, we have
 	$$
 	\inf_{\lambda \in \Lambda} \varphi(\lambda) = \varphi(\bar{\lambda}) \leq \lim_{l \to \infty} \inf \Big\{ \inf_{\lambda \in \Lambda} \varphi_{J_l}(\lambda) \Big\}.
 	$$
 	This leads to a contradiction with \eqref{pp32}. Therefore, as $J_k \to \infty$, $\inf_{\lambda \in \Lambda} \varphi_{J_k}(\lambda) \to \inf_{\lambda \in \Lambda} \varphi(\lambda)$.
 \end{proof}
 
 From proposition \ref{pro2} and Theorem \ref{th1}, we can derive the following theorem.
 \begin{theorem}\label{th2}
 	Let $\{x^{(j)}({\lambda})\}$ be the sequence generated by Step $1$ in Algorithm ADMM-BDA parameterized by $\lambda$, with $s_j = \frac{s}{j+1}$ and $s \in \big(0, \frac{\lambda_{\min}(\mathcal{S})}{{L_\delta}}\big)$. Then, let {$\lambda^{k} \in \arg\min_{\lambda \in \Lambda} \varphi_{J_k}(\lambda)$}. We have:
 	\begin{enumerate}
 		\item[(a)]  Any limit point $(\bar{x}, \bar{\lambda})$ of the sequence $\{x^{(J_k)}({\lambda^{k}}), \lambda^{k}\}$ is a solution to  \eqref{model1}, i.e., $\bar{\lambda} \in \arg\min_{\lambda \in \Lambda} \varphi(\lambda)$ and $\bar{x} = \mathcal{T}_{\bar{\lambda}}(\bar{x})$.\\
 		\item[(b)]  As $J_k \to \infty$, it holds that $\inf_{\lambda \in \Lambda} \varphi_{J_k}(\lambda) = \inf_{\lambda \in \Lambda} \varphi(\lambda)$.
 	\end{enumerate}
 \end{theorem}

 \begin{proof}
 	From Theorem \ref{th1}, we know that there exists a constant $C > 0$ such that for every $\lambda \in \Lambda$
 	$$
 	\|x^{(j)}({\lambda}) - \mathcal{T}_{\lambda}(x^{(j)})\|_{\mathcal{S}} \leq C \sqrt{\frac{1 + \ln(1 + j)}{j^{{1}/{4}}}}.
 	$$
 	Because $\lim_{j \to \infty} \sqrt{\frac{1 + \ln(1 + j)}{j^{{1}/{4}}}} = 0$ and based on the relation \eqref{xlambda}, we know that the condition (a) in Proposition \ref{pro2} is satisfied. Next, it follows from Theorem \ref{th1} that as $ J_k \to \infty $, for every $ \lambda \in \Lambda$, $\varphi_{J_k}(\lambda) = \mathcal{F}(x^{(J_k)}(\lambda), \lambda) \to \varphi(\lambda) $. Therefore, the condition  (b) in Proposition \ref{pro2} is also satisfied.
 	Hence, the claims of the theorem can be followed immediately from Proposition \ref{pro2}.
 \end{proof}
 
 The theorem demonstrates the convergence of the ADMM-BDA algorithm when using a diminishing step size defined as $s_j = \frac{s}{j+1}$. It establishes that every limit point $(\bar{x}, \bar{\lambda})$ from the sequence $\{x^{(J_k)}(\lambda^{k}), \lambda^{k}\}$ corresponds to a solution of the original bilevel problem \eqref{model0}. Additionally, as the number of lower-level iterations $J_k$ increases, the optimal value of the approximate upper-level problem $\inf_{\lambda \in \Lambda} \varphi_{J_k}(\lambda)$ approaches the true optimum $\inf_{\lambda \in \Lambda} \varphi(\lambda)$.
 
\section{Experimental Evaluation of Algorithms on Synthetic and Real-World Data}\label{num}
In this section, we demonstrate the advantages of the proposed ADMM-BDA algorithm over other classical hyperparameter selection methods for sparse optimization, using both simulated and real-world examples.
All the experiments are performed with Microsoft Windows 11 and Python 3.11, and conducted on a personal computer with Intel (R) Core (TM) i5-13500H CPU at $2.70$ GHz and $16$ GB memory.
The benchmark methods for performance comparisons are the following:
\begin{itemize}
	\item [-] \textbf{Grid Search}: Performing a brute-force exploration over a specified uniform hyperparameter space, evaluating every possible combination of hyperparameters.
	\item [-] \textbf{Random Search}: Randomly sampling $30$ points for elastic-net model and $50$ points for generalized-elastic-net model uniformly across each hyperparameter dimension and evaluate them.
	\item [-] \textbf{TPE}: Bayesian optimization method based on the tree-structured parzen estimator with Gaussian mixture models, see Bergstra et al. \cite[Page 115-123]{B2013}.
	\item [-] \textbf{PGM-BDA}: Proximal gradient method based BDA algorithm. Because the algorithm requires the loss term to be smooth, it was only tested in the elastic-net model and does not appear in the comparative experiments of the generalized-elastic-net model.
\end{itemize}
It is important to highlight that all the comparative methods used the ADMM approach to solve the lower-level problem. Additionally, the TPE method implemented in this study can be accessed at the following website: \url{https://github.com/hyperopt/hyperopt}.

\subsection{Experiments with Synthetic Data}
In this part, we evaluate our proposed method ADMM-BDA on an artificially generated synthetic dataset. The dataset contains $200$ training, $20$ validation, and $100$ test samples, each with a feature dimension of $500$. The data was designed to be sparse, with  $5$ non-zero features per sample. The mapping matrix $A$ is generated from a standard Gaussian distribution and subsequently $\ell_2$-normalized. These non-zero values are generated through exponential scaling $2^v$, where $v$ is the scaling exponent, and the noise level is set as $1e-3$. We perform hyperparameter selection experiments on synthetic datasets with varying levels of noise, using both the elastic-net penalty proposed by Zhou and Hastie \cite{ZH2005} and the generalized-elastic-net penalty considered by Xiao et al. \cite{XSD2024}. The experiments incorporate $\ell_q$-norm loss function $\mathcal{Q}(\cdot)$, where $q = 1, 2, \infty$.

In this experiment, to balance computational accuracy and efficiency, we adopt a two convergence criterion. The first criterion is based on the residual reduction of the objective function (ResErr), while the second evaluates the stability of the iterative variables (RelErr), that is,
$$
\text{ResErr} := \frac{\|Ax^k- b\|_2}{1+\|b\|_2}, \quad \text{and}\quad \text{RelErr} := \frac{\|\lambda^k - \lambda^{k-1}\|_2}{1+\|\lambda^{k-1}\|_2},
$$
where $x^k$ is the recovered solution after $k$'s iteration, $\lambda^k$ is the hyperparameter after $k$ iterations.
In the following test, we stop the algorithms' iterative process when $\min\{\text{ResErr},\text{RelErr}\}\leq 1e-4$.
Furthermore, we fix parameters $\zeta = 5e-10$ and $\eta = 1e-10$ to ensure $\mathcal{S}_y$ and $\mathcal{S}_x$ are positive semi-definite matrices, other parameters' values will be determined adaptively at each test.
We evaluate the algorithm's performance using computing time, validation error, and test error. Each test is independently repeated $20$ times, and all results at each time are all recorded.

\subsubsection{Testing for  Elastic-Net Penalized lower-level Problem}
In this part, we employ the elastic-net penalized model to assess the performance of our proposed ADMM-BDA algorithm.
In this case, the hyperparameter selection problem using the elastic-net penalty can be formulated as the following bilevel programming  problem:
\begin{equation}\label{en}
	\begin{aligned}
		\min_{x\in \mathbb{R}^n,\lambda\in\mathbb{R}^2_+}\quad&\frac{1}{2m}\| \tilde{A}x-\tilde{b}\| _2^2\\
		\text{s.t.}\quad&x\in\arg\min_{x\in\mathbb{R}^n}\left\{\frac{1}{2}\|\hat{A}x-\hat{b}\|^2_2+\lambda_1\|x\|_1+\frac{\lambda_2}{2}\|x\|_2^2\right\},
	\end{aligned}
\end{equation}
where $\lambda=(\lambda_1,\lambda_2)^\top$ is a hyperparameter vector.
The upper-level objective utilizes the mean squared error (MSE) loss, defined as $\mathcal{F}(x, \lambda) = \frac{1}{2m} \| \tilde{A}x - \tilde{b} \|_2^2$, where $\tilde{A}$ and $\tilde{b}$ represent the data matrix and target vector from the validation set, respectively. The lower-level problem, on the other hand, is an elastic-net penalized least squares problem, with hyperparameters $\lambda_1$ and $\lambda_2$, where $\hat{A}$ and $\hat{b}$ are derived from the training set.
For the LL problem, we employ an ADMM for its solution. For this purpose, by introducing an auxiliary variable $y := \hat Ax - \hat b$, we can reformulate it as:
\begin{align*}
	\arg\min_{x\in\mathbb{R}^n} \ \Big\{\frac{1}{2}\|y\|_2^2 + \lambda_1\|x\|_1 + \frac{\lambda_2}{2}\|x\|_2^2 \ | \ \hat Ax - \hat b = y\Big\}.
\end{align*}
Given initial point $\{x^{(0)}_l, {y}_l^{(0)}; {z}_l^{(0)}\}$, we can get $x^{(j+1)}_l$ in Step $1.1$ in ADMM-BDA algorithm:
{\small$$
\left\{\hspace{-0.3cm}
\begin{array}{lll}
	&{y}_l^{(j+1)} = \dfrac{{z}_l^{(j)} + \sigma(\hat Ax^{(j)} - \hat b + \eta {y}_l^{(j)})}{1 + \sigma + \sigma\eta}, \\[4mm]
	&{z}_l^{(j+1)} = {z}_l^{(j)} + \sigma\left(\hat Ax^{(j)} - \hat b - {y}_l^{(j+1)}\right), \\[2mm]
	&x_l^{(j+1)} = \operatorname{Prox}_{\lambda_1(\lambda_2+\sigma\zeta)^{-1}\|\cdot\|_1}\left( \dfrac{\sigma(\hat{A}^{\top}\hat{b} + \hat{A}^{\top}{y}_l^{(j+1)}) }{\lambda_2 + \sigma\zeta}\right.\\[2mm]
	&\qquad\qquad\left.-\dfrac{ \hat{A}^{\top}{z}_l^{(j+1)} - \sigma(\zeta I - \hat{A}^{\top}\hat{A})x^{(j)}}{\lambda_2 + \sigma\zeta} \right).
\end{array}
\right.
$$}
Subsequently, we implement Steps 1.2 and 1.3 of the ADMM-BDA algorithm, along with Step 2 to identify the optimal hyperparameters for the model in   \eqref{en}. The performance of the proposed method is then compared with that of other approaches. The parameters employed in this experiment are set as follows: $\sigma = 1 \times 10^{-4}$, $s = 1$, $\mu = 0.7$, and the step size for updating the hyperparameter $\lambda^{k+1}$ is $\alpha=1 \times 10^{-3}$.

In this test, we evaluate the performance of the ADMM-BDA algorithm in hyperparameter selection through comparative experiments.
The results are presented in Figure \ref{fig}, where we plot the solutions against the ground-truth solutions, as well as the test error versus computing time.
As shown in Subfigure \ref{fig:bmo}, the solution $\bar x$ (blue color) obtained by ADMM-BDA closely matches the ground truth $x^t$ (red color), which confirms the algorithm's clear advantage in solution quality.
As shown in Subfigure \ref{fig:Con}, ADMM-BDA achieves the same test error as other compared methods but with significantly lower computing time. The convergence curve consistently remains in the lower-left region, which shows its superior performance in both convergence speed and computational accuracy.
In conclusion, ADMM-BDA provides higher-quality solutions while also offering faster computational efficiency in hyperparameter selection tasks, which highlights its overall advantages in solving the bilevel problem \eqref{en}.

\begin{figure*}[!t]
	\centering
	\subfloat[The solution (blue) derived by ADMM-BDA versus the ground truth (red).]{
		\includegraphics[width=3.2in]{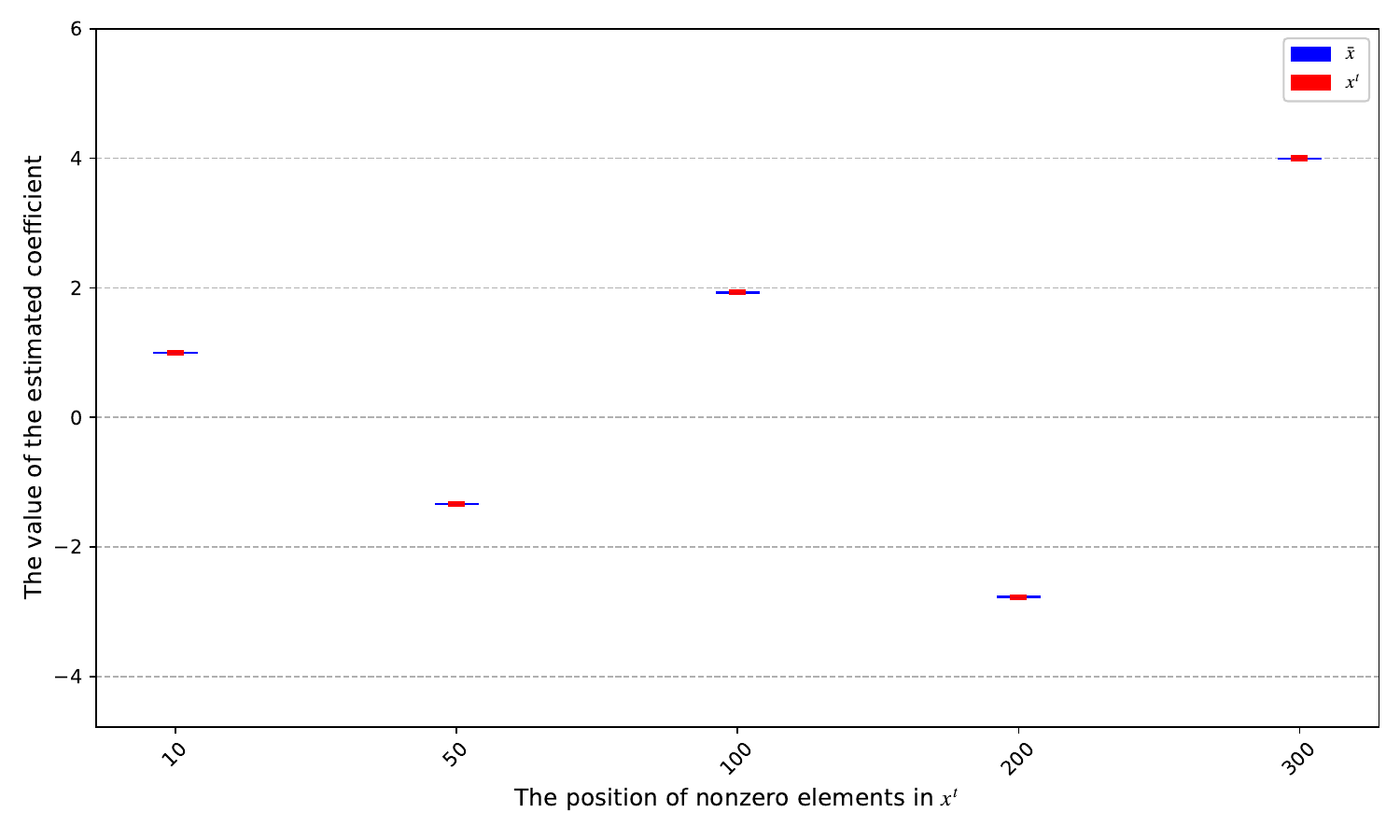}%
		\label{fig:bmo}}
	\hspace{0.5cm}
	\subfloat[Test error versus computing time for ADMM-BDA and comparison methods.]{
		\includegraphics[width=3.2in]{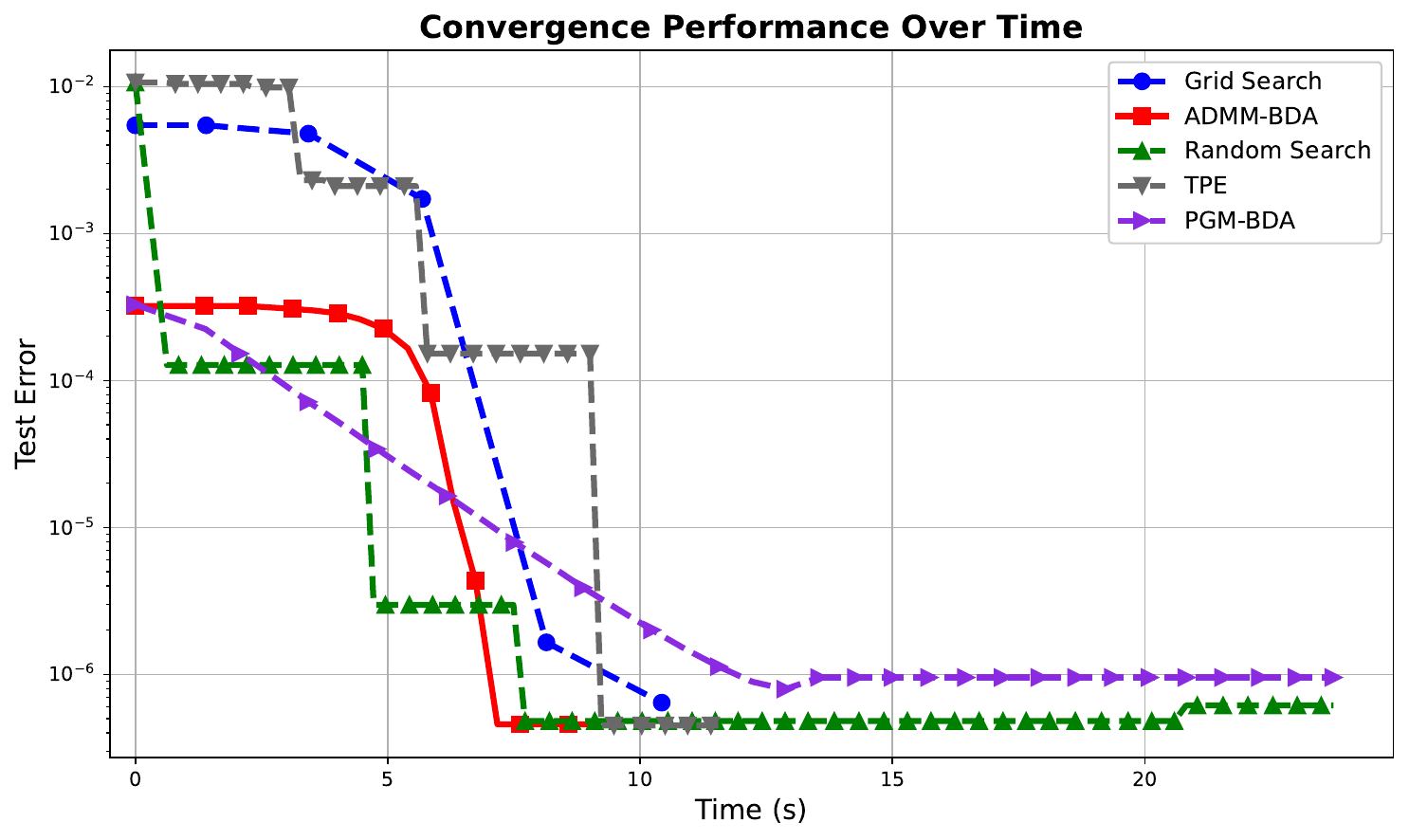}%
		\label{fig:Con}}
	\caption{Numerical performance of algorithms on elastic-net penalized lower-level problem.}
	\label{fig}
\end{figure*}

To provide a clearer comparison of the performance differences among various hyperparameter selection methods,
we report the detailed results in Table \ref{tb1} with regarding to the
computing time in seconds (`Time'), upper-level out-objective function value on the validation (`Val.Err.') and test sets (`Tes.Err.').
In this test, we set $n = 500$, which represents the dimension of the lower-level problem or the number of features.
In the ``Settings" column of Table \ref{tb1}, the values of $I_{tra}$, $I_{val}$, and $I_{tes}$ correspond to the sizes of the training, validation, and test datasets, respectively.
As shown in Table \ref{tb1}, the comparative experiments on the elastic-net penalized model demonstrate that the ADMM-BDA algorithm exhibits significant advantages across the metrics Time, Val.Err, and Tes.Err.
In terms of computational efficiency, ADMM-BDA takes only about $7.8$ seconds, which significantly outperforms the traditional grid search and random search approaches, as well as outperforms both the gradient-based method PGM-BDA and the Bayesian optimization method TPE.
Finally, in terms of solution accuracy, we see that the ADMM-BDA algorithm delivers the best performance on both validation and test errors, with its precision notably surpass that of all other compared methods.
Furthermore, the ADMM-BDA's consistently low variance, which highlights its exceptional stability and robustness.
In summary, from this test, we conclude that ADMM-BDA not only shows distinct advantages in computational efficiency but also consistently outperforms existing methods in terms of solution accuracy and algorithmic stability.

\begin{table*}[!t]
	\renewcommand{\arraystretch}{1.2} 
	\centering
	\caption{Comparative results on synthetic data for the elastic-net penalized lower-level problem.}
	\label{tb1}
	\scalebox{1.4}{
		\begin{tabular}{l|lrrr}
			\hline
			\text{Settings}&\text{Method} & \text{Time}    & \text{Val.Err.}       & \text{Tes.Err.} \\ \hline
			$I_{tra}=200$    &Grid Search     & 20.20 $\pm$ 2.38 & 2.38e-06 $\pm$ 2.95e-06 & 2.10e-06 $\pm$ 1.74e-06 \\
			$I_{val}=20$     &Random Search   & 23.90 $\pm$ 6.41 & 6.00e-07 $\pm$ 5.48e-07 & 7.49e-07 $\pm$ 9.97e-07 \\
			$I_{test}=100$   &TPE             & 9.99 $\pm$ 0.13  & 7.07e-07 $\pm$ 8.55e-07 & 6.86e-07 $\pm$ 6.77e-07 \\
			$n=500$          &PGM-BDA         & 10.20 $\pm$ 1.81 & 9.31e-07 $\pm$ 1.96e-07 & 1.65e-06 $\pm$ 5.68e-07 \\
			&{ADMM-BDA} & \textbf{7.80 $\pm$ 0.16} & \textbf{3.08e-07 $\pm$ 8.68e-08} & \textbf{5.22e-07 $\pm$ 7.91e-08} \\\hline	
		\end{tabular}
	}
\end{table*}

\subsubsection{Testing for Generalized-Elastic-Net Penalized lower-level Problem}
In this subsection, we utilize a generalized-elastic-net penalized model with an $\ell_r$-norm loss function (where $r=1,2,\infty$) in the lower-level problem to assess the performance of the ADMM-BDA algorithm.
In this context, our goal is to derive a sparse solution from the observed data $b \in \mathbb{R}^m$, which is affected by various types of noise. We do this by employing different $q$ values in the loss functions. Specifically, Laplace noise (denoted as `LN') corresponds to $q=1$, Gaussian noise (denoted as `GN') is represented by $q=2$, and uniform noise (denoted as `UN') is associated with $q=\infty$.
The hyperparameter selection problem associated with the generalized elastic-net penalized model can be framed as a bilevel programming problem, as follows:
\begin{equation}\label{gen}
	\begin{aligned}
		\min_{x\in \mathbb{R}^n,\lambda\in \mathbb{R}^2_+}\quad &\frac{1}{2m}\left\| \tilde{A}x-\tilde{b}\right\| _2^2\\
		\text{s.t.}\quad&x\in\arg\min_{x}\left\{\|\hat{A}x-\hat{b}\|_q+\lambda_1\|x\|_1+\frac{\lambda_2}{2}\|x\|_2^2\right\},
	\end{aligned}
\end{equation}
where the notation are consistent with those used in problem \eqref{en}.
Similarly, the iterative framework for obtaining $\{x_l^{(j)}\}$ in Step $1.1$ of ADMM-BDA can be presented as bellow:
{\footnotesize $$
\left\{\hspace{-0.3cm}
\begin{array}{lll}
	&{y}_l^{j+1}=\arg\min_{y}\left\{\|y\|_q + \dfrac{\sigma + \sigma\eta}{2}\left\|y - \dfrac{\sigma (\hat Ax^{(j)}-\hat b) + {z}_l^{(j)} + \sigma\eta {y}_l^{(j)}}{\sigma +\sigma \eta}\right\|_2^2\right\},\\[6mm]
	&{z}_l^{(j+1)} = {z}_l^{(j)} + \sigma\left(\hat Ax^{(j)} - \hat b - {y}_l^{(j+1)}\right), \\[4mm]
	&x_l^{(j+1)} = \operatorname{Prox}_{\lambda_1(\lambda_2+\sigma\zeta)^{-1}\|\cdot\|_1}\left( \dfrac{\sigma(\hat{A}^{\top}\hat{b} + \hat{A}^{\top}{y}_l^{(j+1)}) }{\lambda_2 + \sigma\zeta}\right.\\[4mm]
		&\qquad\qquad\left.-\dfrac{\hat{A}^{\top}{z}_l^{(j+1)} - \sigma(\zeta I - \hat{A}^{\top}\hat{A})x^{(j)}}{\lambda_2 + \sigma\zeta}\right).
\end{array}
\right.
$$}
In the implementation of the iterative scheme in ADMM-BDA for hyperparameter selection in model \eqref{gen}, we set the parameters as $\mu = 0.5$ and $s = 1$. Specifically, we choose $\sigma =  1e-4$ for $q = 1$ or $2$, and $\sigma = 1e-5$ for $q = \infty$.

In this subsection, we conduct a series of comparative experiments to demonstrate the superior performance of ADMM-BDA for hyperparameter selection. The left column of Figure \ref{fig1} presents the sparse solutions obtained by ADMM-BDA compared against the ground truth. Meanwhile, the right column illustrates the performance of each compared algorithm in terms of their computing time.
Firstly, from the figures (a), (c), and (e) in the left column, we see that the ADMM-BDA algorithm exhibits good performance in terms of solution accuracy across three norm loss functions ($r=1, 2, \infty$). Specifically, the sparse solutions (blue color) closely align with the ground truth (red color), which confirms the effectiveness of ADMM-BDA in the context of hyperparameter selection.
Secondly, from figures (b), (d), and (f) in the right column, we see that the curve corresponding to ADMM-BDA consistently positions itself at the bottom in comparison to the other curves, whic indicates that the ADMM-BDA algorithm exhibits better performance in testing error compared to algorithms `Grid Search', `Random Search', `TPE', and `PGM-BDA'.
Finally, upon comparing each row of Figure \ref{fig1}, we see that ADMM-BDA consistently demonstrates superior performance across different types of noise. Specifically, ADMM-BDA achieves the fastest hyperparameter selection while maintaining higher quality sparse solutions in the presence of LN, GN, and UN noise types.

\begin{figure*}[!t]
	\centering
	\subfloat[Sparse solution (blue) derived by ADMM-BDA versus the ground truth (red) under LN.]{
		\includegraphics[width=3.2in]{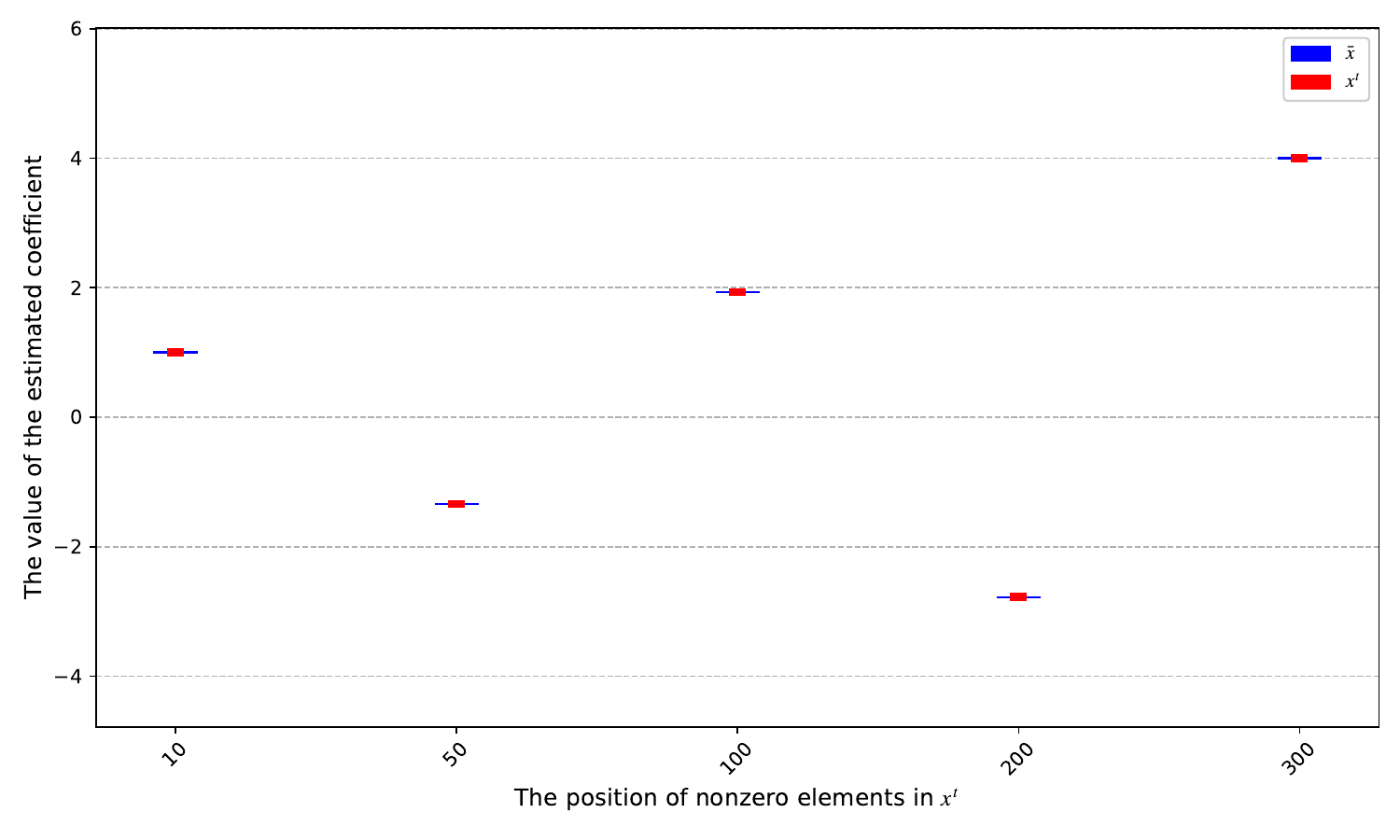}
		\label{fig:bmo1}}
	\hspace{0.5cm}
	\subfloat[Test error versus computing time for ADMM-BDA and comparison methods under LN.]{
		\includegraphics[width=3.2in]{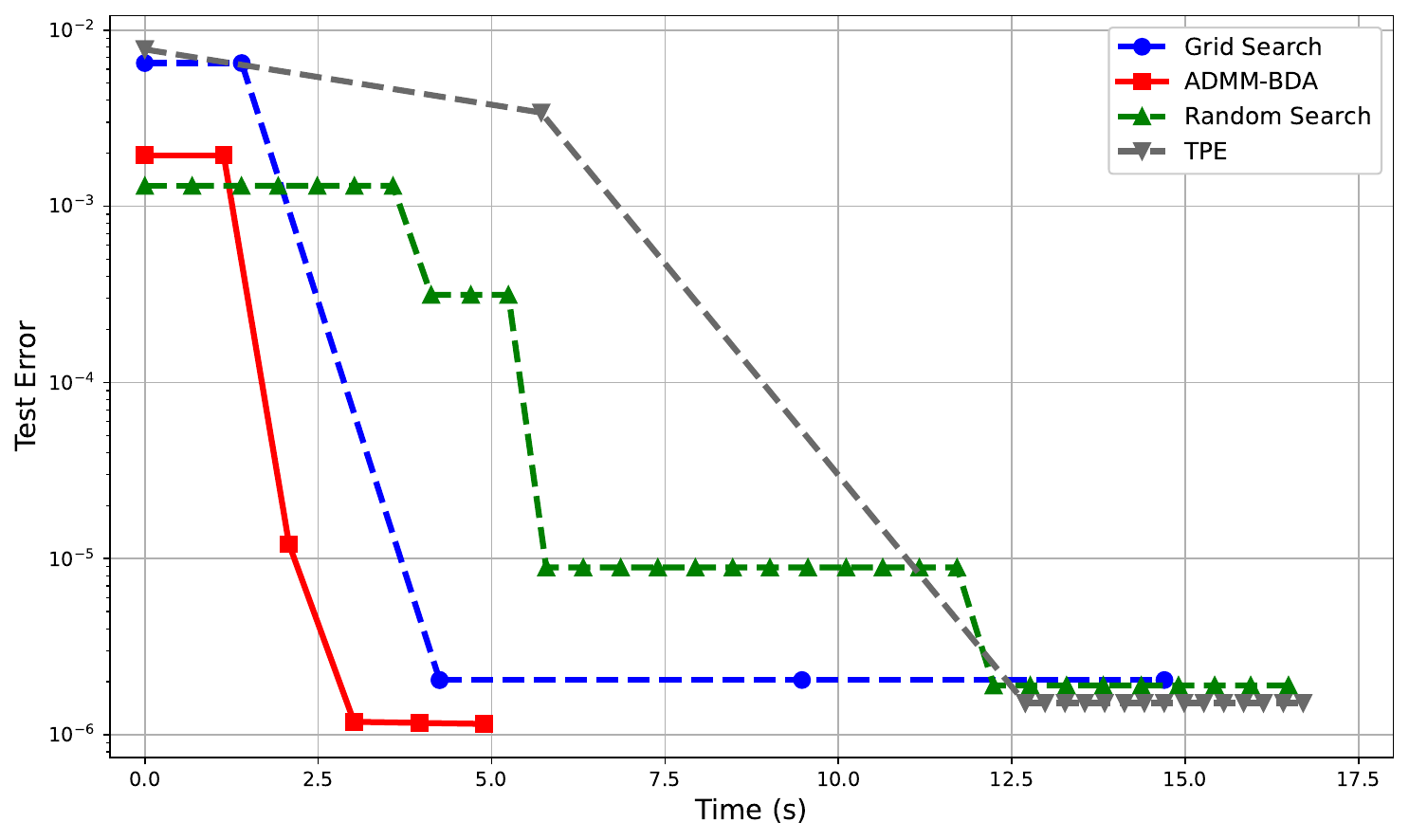}
		\label{fig:con1}}
	
	\vspace{0.3cm}
	
	\subfloat[Sparse solution (blue) derived by ADMM-BDA versus the ground truth (red) under GN.]{
		\includegraphics[width=3.2in]{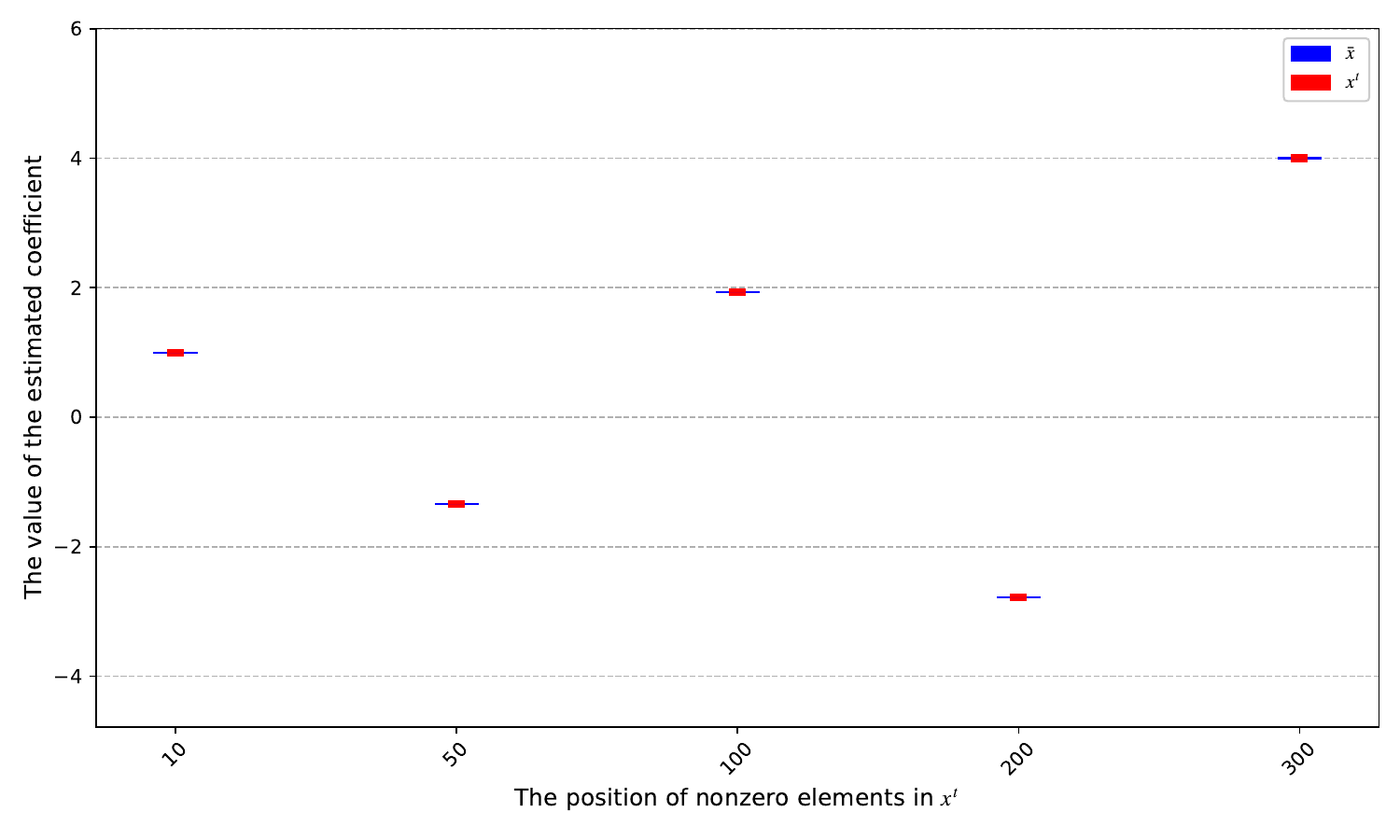}
		\label{fig:bmo2}}
	\hspace{0.5cm}
	\subfloat[Test error versus computing time for ADMM-BDA and comparison methods under GN.]{
		\includegraphics[width=3.2in]{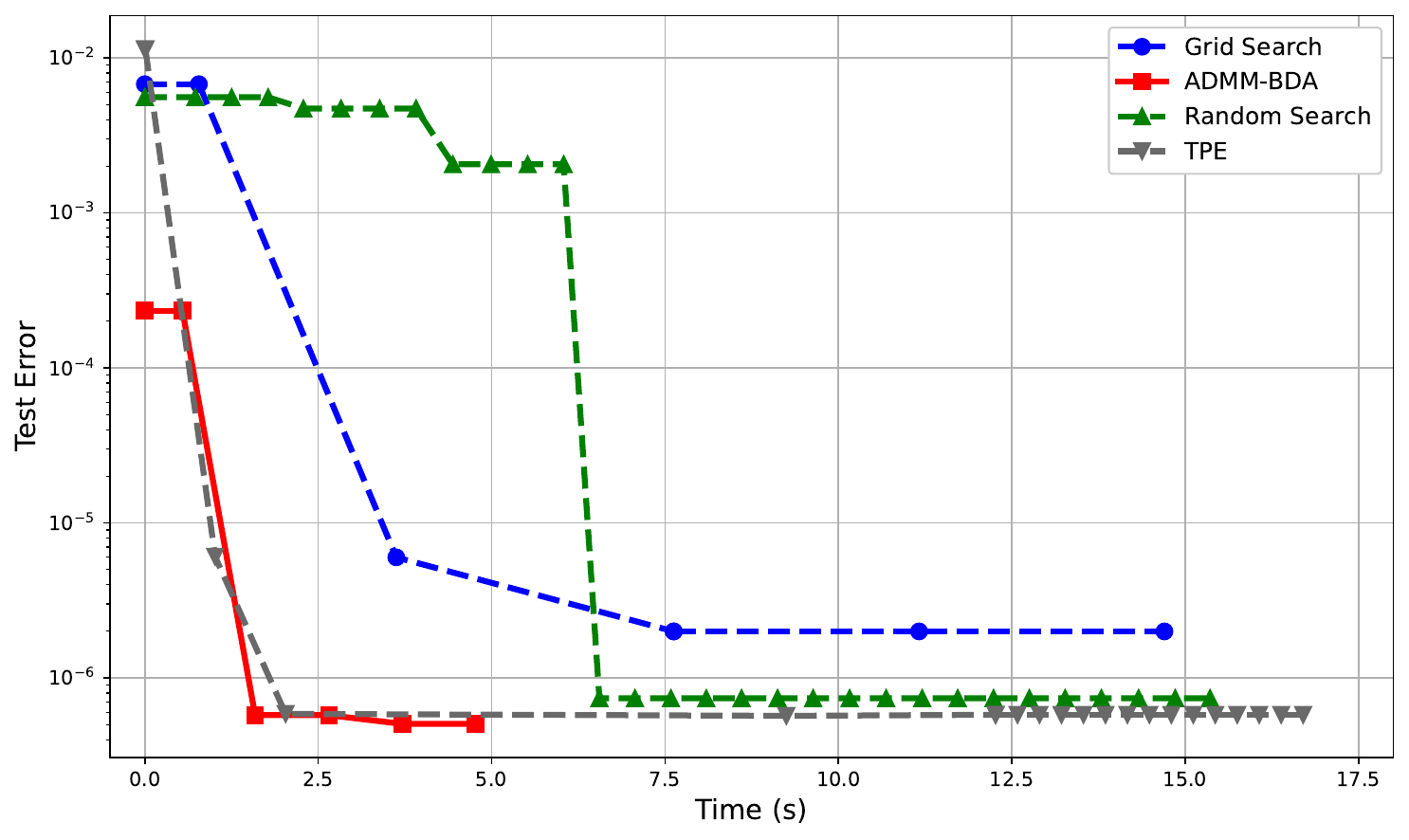}
		\label{fig:con2}}
	
	\vspace{0.3cm}
	
	\subfloat[Sparse solution (blue) derived by ADMM-BDA versus the ground truth (red) under UN.]{
		\includegraphics[width=3.2in]{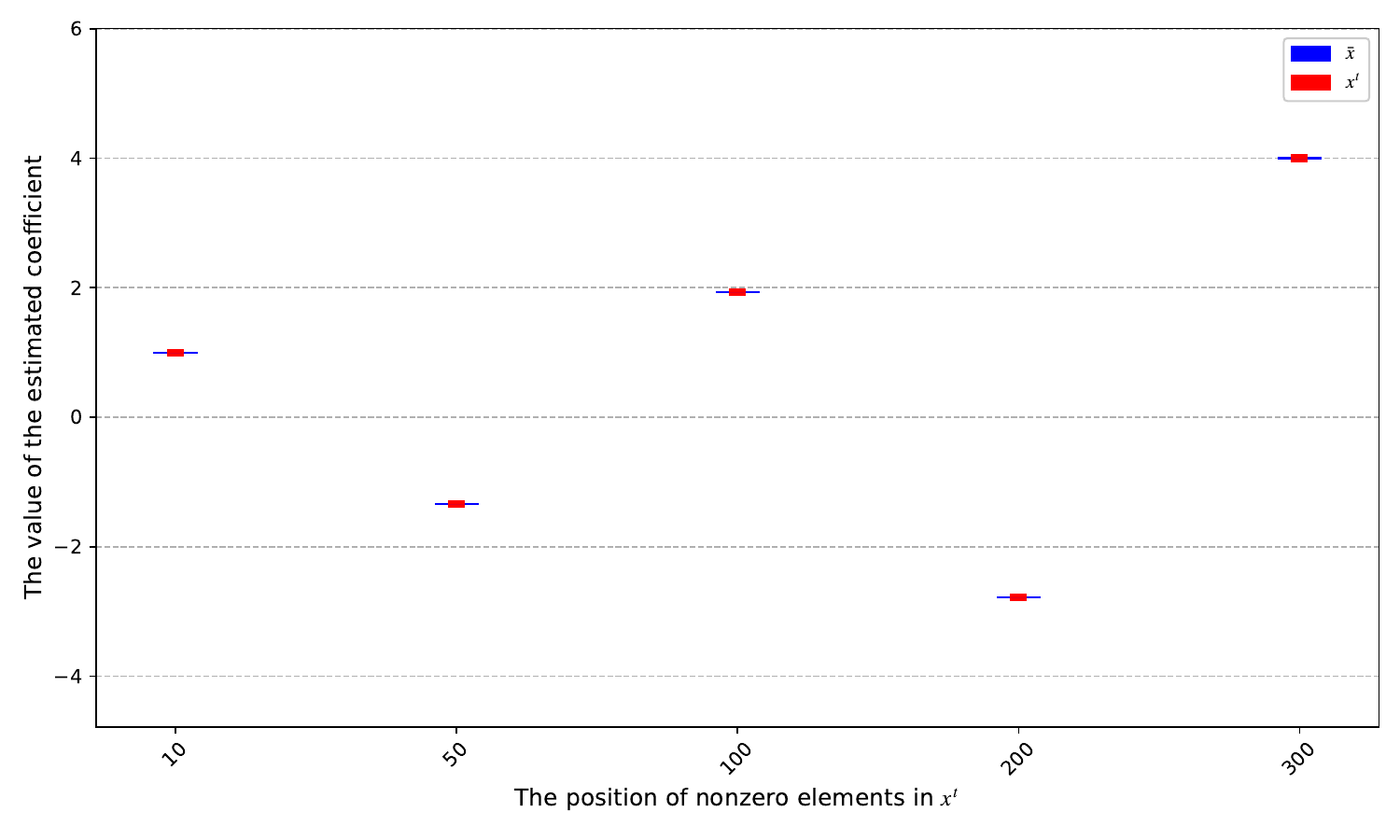}
		\label{fig:bmo3}}
	\hspace{0.5cm}
	\subfloat[Test error versus computing time for ADMM-BDA and comparison methods under UN.]{
		\includegraphics[width=3.2in]{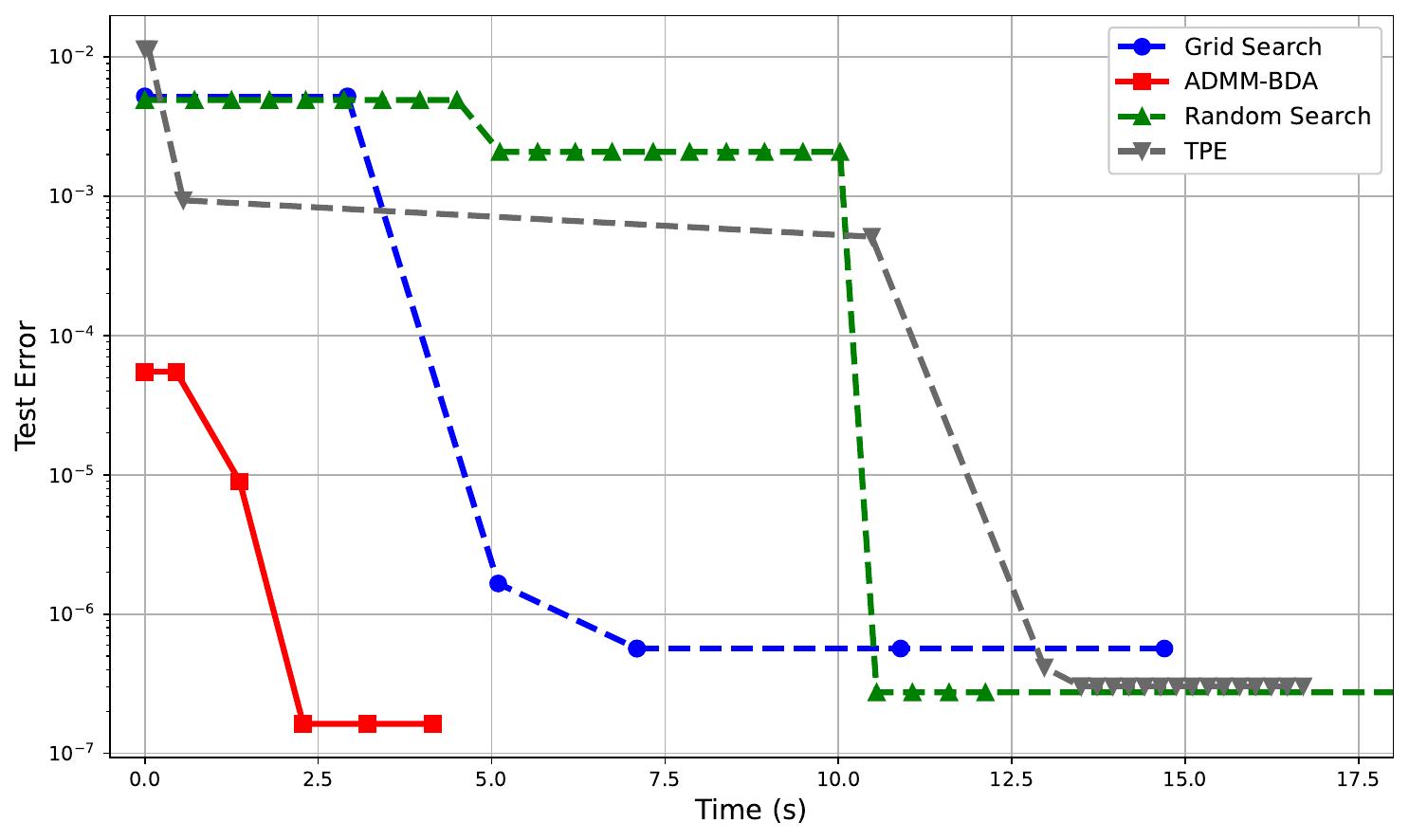}
		\label{fig:con3}}
	
	\caption{Numerical performance of algorithms on generalized-elastic-net penalized lower-level problem across various noise types.}
	\label{fig1}
\end{figure*}
\begin{table*}[!t]	
	\renewcommand{\arraystretch}{1.2} 
	\centering
	\caption{Comparative results of algorithms on synthetic data using the generalized-elastic-net penalized lower-level problem.}
	\label{tb2}
	\scalebox{1.4}{
		\begin{tabular}{llrrr}
			\hline
			\text{Noise} & \text{Method} & \text{Time} & \text{Val.Err.} & \text{Tes.Err.} \\
			\hline
			& Grid Search & 12.33 $\pm$ 0.30 & 1.79e-06 $\pm$ 7.90e-07 & 1.78e-06 $\pm$ 3.79e-07 \\
			LN & Random Search & 17.17 $\pm$ 4.50 & 1.10e-06 $\pm$ 8.28e-07 & 1.27e-06 $\pm$ 8.85e-07 \\
			& TPE & 16.65 $\pm$ 3.69 & 2.05e-05 $\pm$ 8.54e-05 & 2.12e-05 $\pm$ 8.78e-05 \\
			& {ADMM-BDA} & \textbf{5.24 $\pm$ 1.77} & \textbf{3.55e-07 $\pm$ 2.16e-07} & \textbf{1.05e-06 $\pm$ 2.53e-07} \\
			\hline
			& Grid Search & 12.52 $\pm$ 0.53 & 1.79e-06 $\pm$ 1.73e-06 & 1.75e-06 $\pm$ 1.58e-06 \\
			GN & Random Search & 15.70 $\pm$ 2.42 & 1.31e-05 $\pm$ 2.99e-05 & 9.89e-06 $\pm$ 2.21e-05 \\
			& TPE & 18.14 $\pm$ 4.19 & 3.30e-05 $\pm$ 1.37e-04 & 6.60e-05 $\pm$ 2.74e-04 \\
			& {ADMM-BDA} & \textbf{5.39 $\pm$ 1.31} & \textbf{1.79e-07 $\pm$ 5.05e-08} & \textbf{5.05e-07 $\pm$ 6.94e-08} \\
			\hline
			& Grid Search & 14.16 $\pm$ 0.69 & 1.15e-05 $\pm$ 4.95e-05 & 1.71e-05 $\pm$ 7.40e-05 \\
			UN & Random Search & 14.89 $\pm$ 0.41 & 3.02e-07 $\pm$ 5.99e-07 & 2.84e-07 $\pm$ 4.81e-07 \\
			& TPE & 16.21 $\pm$ 5.28 & 1.98e-05 $\pm$ 8.56e-05 & 1.80e-05 $\pm$ 7.75e-05 \\
			& {ADMM-BDA} & \textbf{5.10 $\pm$ 2.84} & \textbf{6.33e-08 $\pm$ 1.32e-08} & \textbf{1.83e-07 $\pm$ 2.37e-08} \\
			\hline
		\end{tabular}
	}
\end{table*}

In order to deliver a more in-depth comparison of algorithm performance under different types of noise, we present the detailed results in Table \ref{tb2}.
This table presents the metrics concerning `Time', `Val.Err.', and `Tes.Err.'.
From this table, we can see that the proposed algorithm ADMM-BDA demonstrates clear advantages across three typical noise scenarios when evaluating the metrics of `Time', `Val.Err.', and `Tes.Err.'.
On the one hand, from the `Time' column, we observe that ADMM-BDA exhibits superior runtime performance, being approximately two to three times faster than its competitors in deriving sparse solutions with comparable accuracy.
On the other hand, from the `Val.Err' and `Tes.Err' columns, we notice that ADMM-BDA exhibits significant superiority in both validation error and test error metrics, as the accuracy has nearly improved by an order of magnitude.
In conclusion, these results show that ADMM-BDA effectively balances computational efficiency and solution accuracy in solving the hyperparameter selection problem, which thereby offers a reliable and efficient approach for sparse optimization under different noise types.

\subsection{  Experiments with Real-World Data}
In this section, we further evaluate the practical effectiveness of ADMM-BDA by using the real-world Bodyfat dataset sourced from the LIBSVM repository \cite{C2011}. It is important to note that the dataset contains $252$ samples, resulting in a feature dimensionality of $680$ after applying third-degree polynomial feature expansion.
In this experiment, we randomly split the data into three segments: a training set, a validation set, and a test set. We then conduct the experiment independently $20$ times with different random partitions to ensure statistical reliability. The detailed results related to `Time', `Val.Err.', and `Tes.Err.' are recorded in Tables \ref{tb3} and \ref{tb4}.
Specifically, Table \ref{tb3} provides a detailed results for each algorithm employed to address the bilevel problem \eqref{en} where an elastic-net penalized least-squares is used in the lower-level problem. Meanwhile, Table \ref{tb4} presents the detailed results from the generalized-elastic-net penalized lower-level problem \eqref{gen}, which incorporates $\ell_r$-norm loss functions to deal with different types of noises.

\begin{table*}[!t]
	\renewcommand{\arraystretch}{1.2} 
	\centering
	\caption{\scriptsize Comparative results on real-world data using elastic-net penalized lower-level problem.}
	\label{tb3}
	\scalebox{1.4}{
	\begin{tabular}{lrrr}
		\hline
		\textbf{Method} & \textbf{Time} & \textbf{Val. Err.} & \textbf{Tes.Err.} \\ \hline
		Grid Search     & 28.10s$\pm$7.3   & 1.90e-04$\pm$ 2.74e-05  & 1.84e-04$\pm$4.61e-05  \\
		Random Search   & 12.59s$\pm$0.47& 1.94e-04$\pm$4.13e-05   & 1.76e-04$\pm$5.49e-05   \\
		TPE       & 13.92s$\pm$2.44    & 1.89e-04$\pm$3.42e-05 & 1.91e-04$\pm$3.65e-05    \\
		PGM-BDA   & 12.50s$\pm$4.33  & 1.52e-04$\pm$6.58e-05 & 2.40e-04$\pm$1.35e-04    \\
		{ADMM-BDA} & \textbf{9.15s$\pm$1.67} & \textbf{1.24e-04$\pm$2.36e-05} & \textbf{1.21e-04$\pm$4.43e-05}  \\ \hline
	\end{tabular}
}
\end{table*}

In the first test, we  use the following parameters' values $\sigma=1e-5$, $s=1e-2$, and $\mu=0.9$ and list the detailed results in Table \ref{tb3}.
We can see from Table \ref{tb3} that, the ADMM-BDA algorithm demonstrates clear advantages in hyperparameter selection problem with using the elastic-net penalized lower-level problem. In terms of the computing time, the ADMM-BDA method requires around $9.15$ seconds, which is nearly $1.5$ times faster than all the other methods evaluated.
In terms of the solution accuracy, the Val.Err and Tes.Err values generated by ADMM-BDA are consistently the smallest, which indicates that ADMM-BDA yields the highest quality sparse solutions among all the methods tested.
In summary, this test demonstrates that our proposed ADMM-BDA method is the most effective solution for the hyperparameter selection problem when an elastic-net penalty is applied in the lower-level problem.

\begin{table*}[!t]
	\renewcommand{\arraystretch}{1.2} 
	\centering
	\caption{\scriptsize Comparative results on real-world data using generalized-elastic-net penalized LL problem under different noise types.}
	\label{tb4}
	\scalebox{1.4}{
	\begin{tabular}{llrrr}
		\hline
		\textbf{Noise} & \textbf{Method} & \textbf{Time} & \textbf{Val.Err.} & \textbf{Tes.Err.} \\
		\hline
		& Grid Search & 14.78 $\pm$ 3.15 & 1.99e-04 $\pm$ 3.24e-05 & 1.91e-04 $\pm$ 5.82e-05 \\
		LN & Random Search & 25.44 $\pm$ 10.50 & 1.77e-04 $\pm$ 4.82e-05 & 1.54e-04 $\pm$ 7.54e-05 \\
		& TPE & 20.02 $\pm$ 2.70 & 2.07e-04 $\pm$ 2.17e-04 & 1.58e-04 $\pm$ 7.03e-05 \\
		& {ADMM-BDA} & \textbf{5.81 $\pm$ 3.96} & \textbf{1.51e-04 $\pm$ 9.21e-05} & \textbf{1.26e-04 $\pm$ 5.65e-05} \\
		\hline
		& Grid Search & 14.46 $\pm$ 5.18 & 2.36e-04 $\pm$ 4.53e-04 & 1.13e-04 $\pm$ 3.55e-05 \\
		GN & Random Search & 17.14 $\pm$ 5.88 & 1.77e-04 $\pm$ 1.39e-04 & 1.34e-04 $\pm$ 1.37e-04 \\
		& TPE & 17.20 $\pm$ 7.19 & 1.65e-04 $\pm$ 4.31e-05 & 1.77e-04 $\pm$ 8.55e-05 \\
		& {ADMM-BDA} & \textbf{4.93 $\pm$ 0.22} & \textbf{8.48e-05 $\pm$ 2.19e-05} & \textbf{6.76e-05 $\pm$ 2.78e-05} \\
		\hline
		& Grid Search & 59.73 $\pm$ 6.85 & 1.29e-04 $\pm$ 6.57e-05 & 1.15e-04 $\pm$ 4.88e-05 \\
		UN & Random Search & 62.96 $\pm$ 9.96 & 1.54e-04 $\pm$ 4.79e-05 & 1.38e-04 $\pm$ 5.25e-05 \\
		& TPE & 74.88 $\pm$ 5.34 & 1.94e-04 $\pm$ 5.42e-05 & 1.89e-04 $\pm$ 6.83e-05 \\
		& {ADMM-BDA} & \textbf{5.42 $\pm$ 3.34} & \textbf{1.10e-04 $\pm$ 5.88e-05} & \textbf{1.03e-04 $\pm$ 3.41e-05} \\
		\hline
	\end{tabular}
}
\end{table*}

In the second test, we set the following parameter values: $\mu = 0.9$ and $s = 1e-4$ and report the detailed results in Table \ref{tb4}. Additionally, for $q = 1$ or $q = 2$, we select $\sigma = 1e-6$, while for $q = \infty$, we opt for $\sigma = 1e-7$.
From the `Time' column in Table \ref{tb4}, we observe that ADMM-BDA requires about $5$ seconds for each test case, making it at least $4$ times, and in some instances up to $12$ times faster than its competitors.
From the last two columns in Table \ref{tb4}, we can see that the values for `Val.Err.' and `Tes.Err.' obtained by each algorithm are comparable to one another.
Especially in the case of UN, we see that the ADMM-BDA is significant faster than other algorithms where it only requires only $5$ seconds to derive higher-quality sparse solutions.
In summary, this test once again highlights the high computational efficiency and good precision advantages of the method in addressing the hyperparameter selection problem using real-world datasets.

\section{Conclusion}\label{con}
In this paper, we addressed the challenge of hyperparameter selection by focusing on a specific class of separable convex non-smooth bilevel optimization problems that do not possess the lower-level singleton property. We proposed an innovative bilevel optimization framework that combines the ADMM and BDA. This framework takes the benefits of both methods: ADMM efficiently tackles the non-smooth lower-level problem, while BDA facilitates the updates of both upper-level and lower-level variables in a coordinated manner.
Theoretically, we eliminated the lower-level singleton assumption and established a convergence analysis framework, proving that our algorithm converges to a solution of \eqref{model0} while approaching the optimal upper-level objective value. We conducted a series of numerical experiments using both synthetic and real-world datasets, demonstrating that our proposed method significantly outperforms existing approaches for bilevel programming problems, particularly those involving elastic-net and generalized-elastic-net penalized lower-level problems. Importantly, this paper not only provided a theoretical analysis for bilevel optimization that extended beyond the lower-level singleton assumption but also presented an effective solution for hyperparameter selection in sparse optimization.

\section*{Conflict of interest statement}

The authors declare that they have no conflict of interest.


%
%

\bibliographystyle{plain}
\bibliography{hcfs1}

@article{CT2005,
  title={Decoding by linear programming},
  author={Candes, Emmanuel J and Tao, Terence},
  journal={IEEE Transactions on Information Theory},
  volume={51},
  number={12},
  pages={4203--4215},
  year={2005},
  publisher={IEEE}
}

@article{ANW2012,
  title={Stochastic optimization and sparse statistical recovery: Optimal algorithms for high dimensions},
  author={Agarwal, Alekh and Negahban, Sahand and Wainwright, Martin J},
  journal={Advances in Neural Information Processing Systems},
  volume={25},
  year={2012}
}

@misc{LHL2020,
  title={Optimization: modeling, algorithm and theory},
  author={Liu, H and Hu, J and Li, Y and Wen, Z},
  year={2020},
  publisher={Higher Education Press Beijing}
}

@article{CF2022,
  title={Bilevel methods for image reconstruction},
  author={Crockett, Caroline and Fessler, Jeffrey A and others},
  journal={Foundations and Trends{\textregistered} in Signal Processing},
  volume={15},
  number={2-3},
  pages={121--289},
  year={2022},
  publisher={Now Publishers, Inc.}
}

@article{BPG2007,
   title={An overview of bilevel optimization},
   author={Colson, Beno{\^\i}t and Marcotte, Patrice and Savard, Gilles},
   journal={Annals of Operations Research},
   volume={153},
   number={1},
   pages={235--256},
   year={2007},
   publisher={Springer}
   }

@inproceedings{LLZ2022,
  title={Optimization-derived learning with essential convergence analysis of training and hyper-training},
  author={Liu, Risheng and Liu, Xuan and Zeng, Shangzhi and Zhang, Jin and Zhang, Yixuan},
  booktitle={International Conference on Machine Learning},
  pages={13825--13856},
  year={2022}
}

@inproceedings{MBB2009,
  title={Nonsmooth bilevel programming for hyperparameter selection},
  author={Moore, Gregory M and Bergeron, Charles and Bennett, Kristin P},
  booktitle={2009 IEEE International Conference on Data Mining Workshops},
  pages={374--381},
  year={2009}
}

@inproceedings{BHJ2006,
  title={Model selection via bilevel optimization},
  author={Bennett, Kristin P and Hu, Jing and Ji, Xiaoyun and Kunapuli, Gautam and Pang, Jong-Shi},
  booktitle={The 2006 IEEE International Joint Conference on Neural Network Proceedings},
  pages={1922--1929},
  year={2006}
}

@article{OTK2021,
  title={On lp-hyperparameter learning via bilevel nonsmooth optimization},
  author={Okuno, Takayuki and Takeda, Akiko and Kawana, Akihiro and Watanabe, Motokazu},
  journal={Journal of Machine Learning Research},
  volume={22},
  number={245},
  pages={1--47},
  year={2021}
}

@inproceedings{GYY2022,
  title={Value function based difference-of-convex algorithm for bilevel hyperparameter selection problems},
  author={Gao, Lucy L and Ye, Jane and Yin, Haian and Zeng, Shangzhi and Zhang, Jin},
  booktitle={International Conference on Machine Learning},
  pages={7164--7182},
  year={2022}
}

@article{KP2013,
  title={A bilevel optimization approach for parameter learning in variational models},
  author={Kunisch, Karl and Pock, Thomas},
  journal={SIAM Journal on Imaging Sciences},
  volume={6},
  number={2},
  pages={938--983},
  year={2013},
  publisher={SIAM}
}

@article{KBH2008,
  title={Classification model selection via bilevel programming},
  author={Kunapuli, Gautam and Bennett, Kristin P and Hu, Jing and Pang, Jong-Shi},
  journal={Optimization Methods \& Software},
  volume={23},
  number={4},
  pages={475--489},
  year={2008},
  publisher={Taylor \& Francis}
}

@article{WL2024,
  title={A fast smoothing newton method for bilevel hyperparameter optimization for SVC with logistic loss},
  author={Wang, Yixin and Li, Qingna},
  journal={Optimization},
  pages={1--30},
  year={2024},
  publisher={Taylor \& Francis}
}

@article{SMK2024,
  title={A Brief Review of Bilevel Optimization Techniques and Their Applications},
  author={Sapre, Mandar S and Kale, Ishaan R},
  journal={Handbook of Formal Optimization},
  pages={1179--1202},
  year={2024},
  publisher={Springer}
}

@inproceedings{P2016,
  title={Hyperparameter optimization with approximate gradient},
  author={Pedregosa, Fabian},
  booktitle={International Conference on Machine Learning},
  pages={737--746},
  year={2016}
}

@article{RFS2019,
  title={Meta-learning with implicit gradients},
  author={Rajeswaran, Aravind and Finn, Chelsea and Kakade, Sham M and Levine, Sergey},
  journal={Advances in Neural Information Processing Systems},
  volume={32},
  year={2019}
}

@inproceedings{BKB2020,
  title={Implicit differentiation of lasso-type models for hyperparameter optimization},
  author={Bertrand, Quentin and Klopfenstein, Quentin and Blondel, Mathieu and Vaiter, Samuel and Gramfort, Alexandre and Salmon, Joseph},
  booktitle={International Conference on Machine Learning},
  pages={810--821},
  year={2020}
}

@inproceedings{FDF2017,
  title={Forward and reverse gradient-based hyperparameter optimization},
  author={Franceschi, Luca and Donini, Michele and Frasconi, Paolo and Pontil, Massimiliano},
  booktitle={International Conference on Machine Learning},
  pages={1165--1173},
  year={2017}
}

@inproceedings{FFS2018,
  title={Bilevel programming for hyperparameter optimization and meta-learning},
  author={Franceschi, Luca and Frasconi, Paolo and Salzo, Saverio and Grazzi, Riccardo and Pontil, Massimiliano},
  booktitle={International Conference on Machine Learning},
  pages={1568--1577},
  year={2018}
}

@inproceedings{SCC2019,
  title={Truncated back-propagation for bilevel optimization},
  author={Shaban, Amirreza and Cheng, Ching-An and Hatch, Nathan and Boots, Byron},
  booktitle={The 22nd international conference on artificial intelligence and statistics},
  pages={1723--1732},
  year={2019}
}

@article{FS2018,
  title={Gradient-based regularization parameter selection for problems with nonsmooth penalty functions},
  author={Feng, Jean and Simon, Noah},
  journal={Journal of Computational and Graphical Statistics},
  volume={27},
  number={2},
  pages={426--435},
  year={2018},
  publisher={Taylor \& Francis}
}

@article{BKM2022,
  title={Implicit differentiation for fast hyperparameter selection in non-smooth convex learning},
  author={Bertrand, Quentin and Klopfenstein, Quentin and Massias, Mathurin and Blondel, Mathieu and Vaiter, Samuel and Gramfort, Alexandre and Salmon, Joseph},
  journal={Journal of Machine Learning Research},
  volume={23},
  number={149},
  pages={1--43},
  year={2022}
}

@article{BT2009,
  title={A fast iterative shrinkage-thresholding algorithm for linear inverse problems},
  author={Beck, Amir and Teboulle, Marc},
  journal={SIAM Journal on Imaging Sciences},
  volume={2},
  number={1},
  pages={183--202},
  year={2009},
  publisher={SIAM}
}

@article{ZH2005,
  title={Regularization and variable selection via the elastic net},
  author={Zou, Hui and Hastie, Trevor},
  journal={Journal of the Royal Statistical Society Series B: Statistical Methodology},
  volume={67},
  number={2},
  pages={301--320},
  year={2005},
  publisher={Oxford University Press}
}

@inproceedings{LMY2020,
  title={A generic first-order algorithmic framework for bi-level programming beyond lower-level singleton},
  author={Liu, Risheng and Mu, Pan and Yuan, Xiaoming and Zeng, Shangzhi and Zhang, Jin},
  booktitle={International Conference on Machine Learning},
  pages={6305--6315},
  year={2020}
}

@book{R2015,
  title={Convex analysis:(pms-28)},
  author={Rockafellar, Ralph Tyrell},
  year={2015},
  publisher={Princeton university press}
}

@article{LST2022,
  title={An augmented Lagrangian method with constraint generation for shape-constrained convex regression problems},
  author={Lin, Meixia and Sun, Defeng and Toh, Kim-Chuan},
  journal={Mathematical Programming Computation},
  volume={14},
  number={2},
  pages={223--270},
  year={2022},
  publisher={Springer}
}

@article{ZLZ2021,
  title={Plug-and-play image restoration with deep denoiser prior},
  author={Zhang, Kai and Li, Yawei and Zuo, Wangmeng and Zhang, Lei and Van Gool, Luc and Timofte, Radu},
  journal={IEEE Transactions on Pattern Analysis and Machine Intelligence},
  volume={44},
  number={10},
  pages={6360--6376},
  year={2021},
  publisher={IEEE}
}

@book{BC2011,
	author = {Bauschke, Heinz and Combettes, Patrick},
	year = {2011},
	month = {01},
	pages = {468},
	title = {Convex Analysis and Monotone Operator Theory in Hilbert Space},
	isbn = {978-1-4419-9466-0},
	publisher = {Springer} 
}

@article{C2005,
	title={Proximal point algorithm controlled by a slowly vanishing term: applications to hierarchical minimization},
	author={Cabot, Alexandre},
	journal={SIAM Journal on Optimization},
	volume={15},
	number={2},
	pages={555--572},
	year={2005},
	publisher={SIAM}
}

@inproceedings{B2013,
title={Making a science of model search: Hyperparameter optimization in hundreds of dimensions for vision architectures},
author={Bergstra, James and Yamins, Daniel and Cox, David},
booktitle={International Conference on Machine Learning},
pages={115--123},
year={2013}
}

@article{C2011,
	title={LIBSVM: A library for support vector machines},
	author={Chang, Chih-Chung and Lin, Chih-Jen},
	journal={ACM Transactions on Intelligent Systems and Technology},
	volume={2},
	number={3},
	pages={1--27},
	year={2011},
	publisher={Acm New York, NY, USA} 
}

@article{DZL2023,
  title={An efficient semismooth Newton method for adaptive sparse signal recovery problems},
  author={Ding, Yanyun and Zhang, Haibin and Li, Peili and Xiao, Yunhai},
  journal={Optimization Methods and Software},
  volume={38},
  number={2},
  pages={262--288},
  year={2023},
  publisher={Taylor \& Francis}
}

@article{XSD2024,
  author = {Xiao, Yyunhai and Shen, Jian and Ding, Yanyun and Shi, Mengjiao and Li, Peili},
  title = {A fast and effective algorithm for sparse linear regression with $\ell_p$-norm data fidelity and elastic net regularization},
  journal = {Journal of Nonlinear and Variational Analysis},
  volume = {8},
  pages = {433-449},
  year = {2024}
}

\vfill

\end{document}